\documentclass[11pt]{article}

\usepackage{amsmath}
\usepackage{epsfig, amssymb, amsfonts}
\usepackage{amsthm}
\usepackage{amstext}
\usepackage{amsopn}
\usepackage{mathrsfs}
\usepackage{subfigure}
\usepackage{graphicx}
\usepackage[left=2.3cm,top=2cm,right=2.3cm,bottom=2cm, nohead,foot=1cm]{geometry}
\usepackage[makeroom]{cancel}
\usepackage{color}
\usepackage{enumerate}
\usepackage{comment}

\numberwithin{equation}{section}

\newtheorem{theorem}{Theorem}[section]

\newtheorem{lemma}[theorem]{Lemma}

\newtheorem{corollary}[theorem]{Corollary}

\newtheorem*{theorem*}{Theorem}
\newtheorem*{proposition*}{Proposition}

\theoremstyle{definition}
\newtheorem{definition}[theorem]{Definition}
\newtheorem{remark}[theorem]{Remark}

\newcommand{\R}{{\mathbb R}}

\newcommand{\Z}{{\mathbb Z}}

\newcommand{\sbf}{{s}}
\newcommand{\bbf}{{b}}

\newcommand{\ds}{{\displaystyle}}

\newcommand{\TS}{{\mkern-1mu\times\mkern-1mu}}

\newcommand{\BA}{{\widehat{\beta}}}

\newcommand{\Try}{{\triangleleft}}

\date{\vspace{-9ex}}

\title{Nested critical points for a directed polymer on a disordered diamond lattice  \vspace{.4cm}}

\date{  }

  \author{\textbf{Tom Alberts}\footnote{Department of Mathematics,
 University of Utah: {\tt   alberts@math.utah.edu}}  \hspace{.3cm} \& \hspace{.3cm} \textbf{Jeremy Clark}\footnote{Department of Mathematics,  University of Mississippi: {\tt
jeremy@olemiss.edu}} \vspace{.5cm} }

\begin{document}
\maketitle

\begin{abstract} We consider a model for a directed polymer in a random environment defined on a hierarchical diamond lattice in which i.i.d.\ random variables are attached to the lattice bonds.  Our focus is on scaling schemes in  which a size parameter $n$, counting the number of hierarchical layers of the system,  becomes large as the inverse temperature $\beta$ vanishes. When $\beta$ has the form $\BA/\sqrt{n}$ for a parameter $\BA>0$, we show that there is a cutoff value $0 < \kappa < \infty$ such that as $n \to \infty$ the variance of the normalized partition function tends to zero for $\BA\leq \kappa $ and grows without bound  for $\BA > \kappa $. We obtain a more refined description of the border between these two regimes by setting the inverse temperature to $\kappa/\sqrt{n} + \alpha_n$ where $0 < \alpha_n \ll 1/\sqrt{n}$ and analyzing the asymptotic behavior of the variance. We show that when $\alpha_n = \alpha (\log n-\log \log n)/n^{3/2}$ (with a small modification to deal with non-zero third moment) there is a similar cutoff value $\eta$ for the parameter $\alpha$ such that when $\alpha < \eta$ the variance goes to zero and grows without bound when $\alpha > \eta$. Extending the analysis yet again by probing around the inverse temperature $\kappa/\sqrt{n} + \eta (\log n-\log \log n)/n^{3/2}$ we find an infinite sequence of nested critical points for the variance behavior of the normalized partition function. In the subcritical cases $\BA \leq \kappa$ and $\alpha \leq \eta$ this analysis is extended to a central limit theorem result for the fluctuations of the normalized partition function.
\end{abstract}

\section{Introduction}

\subsection{Preliminary discussion}

A \textit{directed polymer in a  random environment} is a statistical mechanical model for a random, self-avoiding path (the polymer) whose law is skewed by a second layer of randomness identified with ``impurities" in a surrounding medium (the environment).  In these models, the environmental  impurities are defined by the realizations of i.i.d.\  random variables spread out over the spatial lattice traversed by the polymer, and the fundamental question is to what extent realizations of the  environment  determine the polymer's trajectory.     The degree of environmental  influence  on the polymer's law  is  modulated in a direct way by an inverse temperature parameter $\beta\geq 0$ (defining a Gibbs measure) and filtered indirectly through the  path combinatorics of the lattice structure on which the polymer is built.   From a mathematical perspective, directed polymers in random media provide interesting probabilistic models since they can exhibit critical behavior as a function of the temperature in the limit that the polymer becomes long.   \vspace{.2cm}

The most  studied  spatial setting for directed polymers is the rectangular lattice $\Z^{+}\times \Z^{d}$ in which the  component $\Z^{+}$ formally plays the role of ``time" for a stochastic process taking values in $\Z^d$; see, for instance,~\cite{Imbr,Bolth,Carmona,Vargas,Sepp,alberts} and the review~\cite{Comets}.  Recently a few authors~\cite{lacoin,US} have focused  on directed polymers embedded on~\textit{diamond graphs} (also referred to as  \textit{diamond lattices}), which are a family of   recursively defined graphs  providing a network of directed paths between two root vertices, $A$ and $B$  (find a  diagram  and the precise definition below).   Before the topic of directed polymers in random media had begun to fully develop as a topic in probability, directed polymers on diamond graphs  were conceived of  in physics literature~\cite{Derrida,Gardner,Cook} as a somewhat simplified but still challenging  alternative setting for developing an understanding of random polymer models.   The simplifying feature of the diamond graph models in comparison to the rectangular lattice is that the distributions of some statistical mechanical quantities, such as the partition function,  obey recursive relations as a function of the system size (this will be clear from~(\ref{Induct}) below).   Other closely related  statistical mechanical phenomena  studied in the context of  diamond graphs include:  spin systems~\cite{Griffiths}, percolation~\cite{HamblyII}, conductance models~\cite{Spohn,Wehr,Hambly},  and especially  pinning models~\cite{Hakim,Giacomin,Garel,GLT,lacoin3}.


\subsection{The model and main results}\label{SecMain}

Given parameter values $b\in \mathbb{N}$ and   $\sbf \in \mathbb{N}$, the diamond graphs $D_{n}$ are constructed for   integers $n\geq 0$ through the following iterative recipe:
\begin{itemize}
\item $D_{0}$  is the graph consisting of a single edge connecting  ``root" vertices $A$ and $B$.

\item  The $n^{th}$ diamond graph, $D_{n}$, is constructed by replacing each edge on $D_{n-1}$ by a ``diamond" formed by
 $\bbf $ branches that are each split into $\sbf $ segments.
\end{itemize}

\begin{center}
\includegraphics[scale=.5]{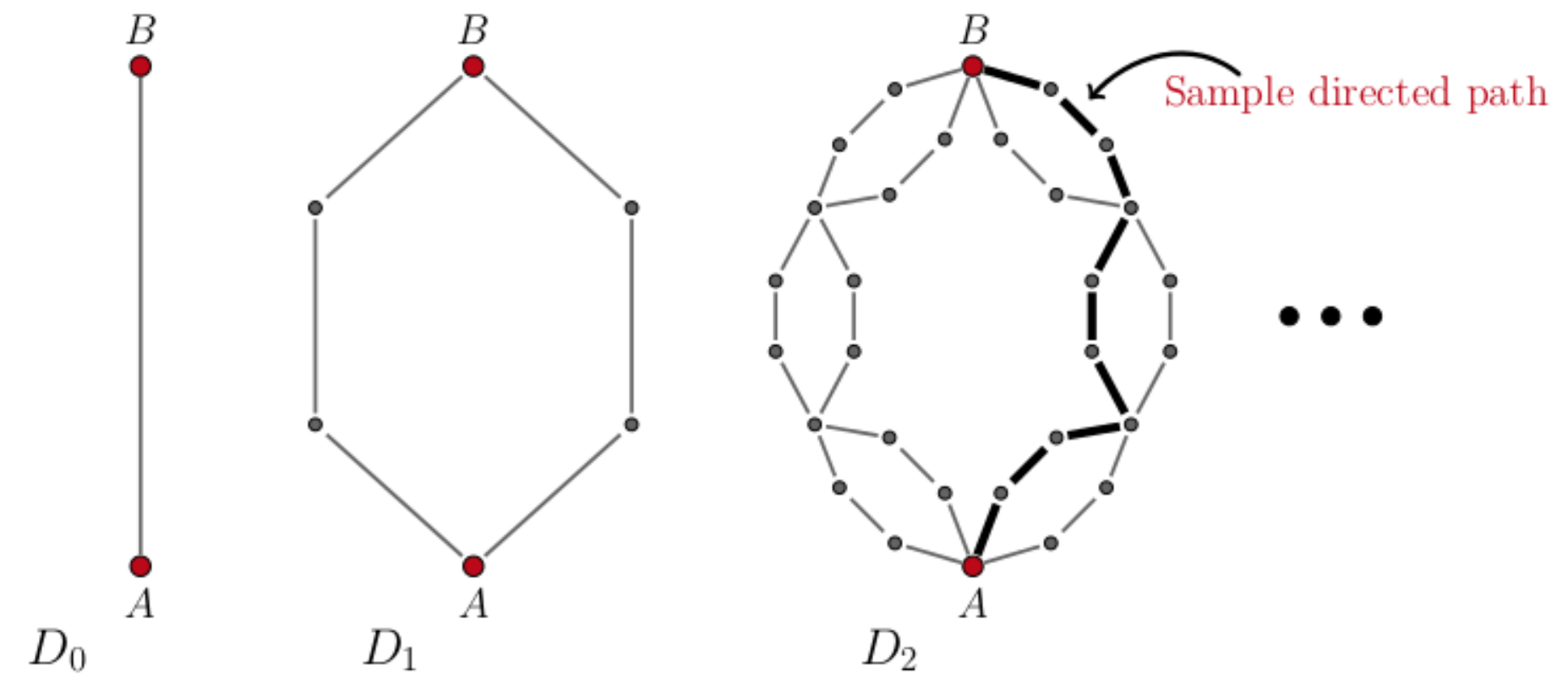}\\
\small The first three diamond graphs determined by $\bbf=2$  and $\sbf=3$.
\end{center}
We denote the set of edges on $D_{n}$ by $E_{n}$.  The $n^{th}$ diamond graph defines a discrete set of directed paths between the root vertices $A$ and $B$, which we denote by $\Gamma_{n}$.   The term ``directed" means that the paths move progressively towards the destination $B$ without self-intersections.  There are $(bs)^{n}$ edges and $ b^{\frac{s^{n}-1}{s-1}  }$ directed paths on $D_{n}$.  This article will be devoted exclusively to the case $b=s$ since the cases $b>s$ and $s<b$ behave differently and require other forms of analysis.

Let  $\omega_{a}$ be i.i.d.\ random variables labeled by the edge set, $E_{n}$,   having mean zero, variance one, and  finite exponential moments: $  \mathbb{E}\big[ e^{\beta \omega_{a}  }  \big] \, < \, \infty\,   $.
  The  partition function $Z_{n}(\beta)$ is defined by
$$Z_{n}(\beta) \, := \,\frac{1}{|\Gamma_{n}|  }\sum_{p\in \Gamma_{n}  }e^{\beta  H_{n}^{\omega}(p) } \hspace{1.2cm}\text{for}\hspace{1.2cm}H_{n}^{\omega}(p) \, :=  \, \sum_{a\Try p} \omega_{a} \, ,
     $$
where   the notation $a\Try p$  means that the edge $a\in E_{n}$ lies along a path $p\in \Gamma_{n}$.  The partition function normalizes random probability measures $\mu^{(\omega)}_{\beta, n}$ on the path space $p\in \Gamma_{n}$ through the Gibbs density:
$$\mu^{(\omega)}_{\beta, n}\big(p) \,   = \, \frac{  e^{\beta  H_{n}^{\omega}(p) }    }{  Z_{n}(\beta)   } \,   .   $$
The above  differs from the partition function in~\cite{lacoin,US}, where the disorder variables $\omega_{a} $ are placed at the vertices of the graph rather than the edges. Disorder for directed polymer models is commonly discussed through the normalized partition function:
\begin{align}\label{ExpW}
W_{n}(\beta)\, := \, \frac{  Z_{n}(\beta)  }{  \mathbb{E}\big[  Z_{n}(\beta)   \big] }\, ,
\end{align}
which can be written as a sum of products
\begin{align}\label{ExpWII}
W_{n}(\beta)\, := \, \frac{1}{|\Gamma_{n}|  }\sum_{p\in \Gamma_{n}  }\prod_{a\Try p} \mathbf{E}(\beta; a   )\hspace{.8cm}\text{for}\hspace{.8cm}\mathbf{E}(\beta; a   ):= \frac{e^{\beta  \omega_{a}  }  }{ \mathbb{E}[e^{\beta  \omega_{a}  }]  }\,.
\end{align}
The inductive construction of the diamond graphs $D_{n}$ implies that there is a recursive relation between the distributions of $W_{n}(\beta)$ and $W_{n+1}(\beta)$;  if $W_{n}^{(i,j)}(\beta)$  are independent copies of  $W_{n}(\beta)$, then there is the following equality in distribution
\begin{align}\label{Induct}
W_{n+1}(\beta)  \, \stackrel{ d }{ =}  \,  \frac{1}{b}\sum_{i=1}^{b} \prod_{j=1}^{b}  W_{n}^{(i,j)}(\beta)  \, .
\end{align}

The two theorems below make analogous statements to those found in Theorem~2.4 of~\cite{US}, concerning the same hierarchical diamond lattice model except with disorder placed on sites. The first theorem characterizes the large $n$ behavior of the variance of the random variables $W_{n}\big(\widehat{\beta} /\sqrt{n}\big)$ where $\BA$ is a positive parameter.  The consideration of limits in which the size $n$ of the system  and the inverse temperature $\beta$ are simultaneously scaled is characteristic of studies of the \textit{intermediate disorder regime}~\cite{alberts,Alberts_Ortgiese, Flores,US,Carav2,DeyZyg}, which effectively magnifies the system behavior in a shrinking (with $n\gg 1$) region of $\beta$ around  the critical point $\beta_{c}$ at which the system transitions from weak disorder to strong disorder. As we will explain later it is not clear that the scaling above is truly accessing the intermediate disorder regime for the edge model, but nonetheless it produces interesting results.

\begin{theorem}[Critical point]\label{ThmOld}
Define the  cut-off value \[\kappa_{b}:= \left( \frac{2}{b-1}  \right)^{1/2}.\]  As $n\rightarrow\infty $, we have the following $\widehat{\beta}$-dependent behavior in the variance of $W_{n}\big( \widehat{\beta}/\sqrt{n} \big) $:
\begin{align}\label{Inf}
\text{}\hspace{1cm}\textup{Var}\bigg(W_{n}\Big(\frac{\widehat{\beta}}{\sqrt{n}}\Big) \bigg) \quad  \xrightarrow{n \to \infty}  \quad \begin{cases} 0\,, & \quad   0\leq \widehat{\beta}  \leq   \kappa_{b} \, ,  \vspace{.2cm}    \\  \infty\,,  & \quad   \,\,\,\,\,\,\widehat{\beta} > \kappa_{b} \, . \end{cases}  
\end{align}
\end{theorem}

Thus the values of the random variables $ W_{n}\big(\widehat{\beta}/ \sqrt{n}\big) $ are concentrated near $1$ when  $\widehat{\beta}  \leq  \kappa_{b}$.   The following theorem characterizes the fluctuations away from $1$ for large $n$.

\begin{theorem}[Central limit theorems]\label{ThmOldII}
  When  $\widehat{\beta}  <  \kappa_{b}$  we have the following weak convergence:
\begin{align}\label{WConv}
 \sqrt{n}\bigg(W_{n}\Big(\frac{\widehat{\beta}}{\sqrt{n}}\Big) -1\bigg)    \quad  \stackrel{\mathcal{L}}{\Longrightarrow} \quad  \mathcal{N}\left( 0,\, \left( \widehat{\beta}^{-2} - \kappa_{b}^{-2} \right)^{-1} \right) \,.
\end{align}
At the critical value $\widehat{\beta}=\kappa_{b}$, the limit result becomes:
\begin{align}\label{WConvII}
  \sqrt{\log n}\left(W_{n}\Big(\frac{\kappa_{b} }{\sqrt{n}}-\frac{ \tau\kappa_{b}^2 }{2n} \Big) -1\right)    \quad  \stackrel{\mathcal{L}}{\Longrightarrow} \quad  \mathcal{N}\Big( 0,\, \frac{6}{b+1}\Big) \,,
\end{align}
where $\tau$ is the third moment of $\omega$.
\end{theorem}

The  skew term $\tau$ in the scaling $\beta_{n}=  \kappa_{b} /\sqrt{n}- \tau\kappa_{b}^2/2n$  from~(\ref{WConvII})  is an adjustment ensuring that $\tau$ does not appear in the variance of $W_{0}(\beta_{n})$  up to order $\mathit{O}( 1/n^2)$:
$$\textup{Var}\big( W_{0}(\beta_{n})\big)\,=\,   \textup{Var}\bigg(  \frac{e^{\beta_{n}  \omega  }  }{ \mathbb{E}[e^{\beta_{n}  \omega  }]  } \bigg)\,=\,\frac{\kappa_{b}^2}{n}\,+\,\mathit{O}\Big( \frac{1}{n^2} \Big)  \,.  $$
With this large $n$ scaling,  the above theorem states that $W_{n}(\beta_{n})$ is roughly a Gaussian with mean one and variance inversely proportional to $\log n$.   Moreover, as a consequence of~(\ref{Inf}), if $\beta_{n}$ is scaled so that
$$ \textup{Var}\big( W_{0}(\beta_{n}) \big)\,=\,\frac{\kappa_{b}^2+\epsilon}{n}\,+\,\mathit{o}\Big( \frac{1}{n} \Big)  \,    $$
for any $\epsilon>0$, then $\textup{Var}\big(W_{n}(\beta_{n})\big)$ grows without bound.   Consequently, it is interesting to consider scalings  $\beta_{n}$ such that $\textup{Var}\big( W_{0}(\beta_{n}) \big)=  \kappa_{b}^2/n+\alpha_{n}  $  for $0<\alpha_{n}\ll 1/n$ to gain an understanding of the border between the zero variance and the infinite variance regimes.  We find that there is a family of scalings $\beta_{n, \epsilon}^{(m)}$ that are  well-suited to the model:

\begin{definition}\label{DefHere}
Define the parameter $\eta_b$ by
\[\eta_{b}:=\frac{b+1}{3(b-1) }.\]
Then let $\ell:\R^{+}\rightarrow \R^{+}$ be defined by $\ell_{x}:=\log(1+x)-\log\log(1+x)$ and $\ell^{m}$ be the $m$-fold composition of $\ell$.  For $\epsilon\in \R^+ $ and $m\in \mathbb{N}$ define
\begin{align*}
\beta_{n, \epsilon}^{(m)}\, :=\,  \frac{ \kappa_{b}}{\sqrt{n}}  \,-\, \frac{ \tau\kappa_{b}^2 }{2n}  \,+\, \frac{\kappa_{b}}{2n^{\frac{3}{2}}}\bigg( \eta_{b}\sum_{k=1}^{m-1}    \ell^{k}_{n} \,+\, \epsilon  \ell^{m}_{n}   \bigg) \,.
\end{align*}
  Note, of course, that $\beta_{n, \eta_{b} }^{(m)}=\beta_{n, 0 }^{(m+1)}$. For $\epsilon < \eta_b$ we also define
\[\upsilon_{b}(\epsilon):=\frac{  \kappa_{b}^2  }{\eta_{b}-\epsilon }. \]
\end{definition}
\begin{remark}\label{RmkExp}
The scale factor $\beta_{n, \epsilon}^{(m)}$ and the variance $ \textup{Var}\big( W_{0}(\beta_{n,\epsilon}^{(m)}) \big)$   have the asymptotic forms
\begin{align}\label{varrho}
\beta_{n, \epsilon}^{(m)}\, =\, & \frac{ \kappa_{b}}{\sqrt{n}}  \,-\, \frac{ \tau\kappa_{b}^2 }{2n}  \,+\, \frac{\kappa_{b}}{2n^{\frac{3}{2}}}\Big( \eta_{b}\big(\log n-\log^m n    \big)  \,+\, \epsilon \big(\log^m n-\log^{m+1} n    \big)   \Big)\, +\,\mathit{O}\Big(\frac{1}{n^{\frac{3}{2}}}\Big)\nonumber \\
  \textup{Var}\big( W_{0}\big(\beta_{n,\epsilon}^{(m)}\big) \big)  \,  =  &\, \frac{ \kappa_{b}^2}{n}\,+\,\frac{\kappa_{b}^2}{n^2}\bigg(\eta_{b}\sum_{j=1}^{m-1}\ell^{j}_n \,+\,\epsilon\ell^{m}_{n}\bigg)\, +\,\mathit{O}\Big(\frac{1}{n^2}\Big)\,  \\  =  &\, \frac{ \kappa_{b}^2}{n}\,+\,\frac{\kappa_{b}^2}{n^2}\Big( \eta_{b}\big(\log n-\log^m n    \big) \,+\,\epsilon\big(\log^m n-\log^{m+1} n  \big)\Big)\, +\,\mathit{O}\Big(\frac{1}{n^2}\Big)\, , \nonumber
\end{align}
where $\log^m$ is the $m$-fold composition of the $\log$ function.  Notice that~(\ref{varrho}) is of the general type of asymptotics referred to in the discussion above Definition~\ref{DefHere}. 

\end{remark}

The term $\upsilon_b(\epsilon)$ is a limiting variance in the following theorem, which is this article's main result.

\begin{theorem}\label{ThmMain}Fix $m\in \mathbb{N}$ and define $\beta_{n, \epsilon}^{(m)}> 0  $ as above.
\begin{enumerate}[(i)]
\item As $n\rightarrow \infty$ there is critical point at $\eta_{b}$ in the asymptotic behavior of the the variance:
$$\text{}\hspace{1cm}\textup{Var}\Big(W_{n}\big(\beta_{n, \epsilon}^{(m)}\big)\Big)   \quad  \xrightarrow{n \to \infty} \quad \begin{cases} 0\,, & \quad   0\leq \epsilon  \leq \eta_{b}\, ,  \vspace{.2cm}    \\  \infty \,,  & \quad   \,\,\,\,\,\,\epsilon >\eta_{b} \, . \end{cases}  $$

\item Moreover, when $\epsilon<\eta_{b}$, the deviation of $W_{n}\big(\beta_{n, \epsilon}^{(m)}\big) $ from one can be characterized by the weak convergence
\begin{align*}
  \sqrt{ \ell^{m}_{n} }\left(W_{n}\big(  \beta_{n, \epsilon}^{(m)} \big) -1\right)    \quad  \stackrel{\mathcal{L}}{\Longrightarrow} \quad  \mathcal{N}\big( 0,\,\upsilon_{b}(\epsilon) \big).
\end{align*}
When $\epsilon=\eta_{b}$ the above convergence holds with $ \ell^{m}_n$ and $\upsilon_{b}(\epsilon) $ replaced by $ \ell^{m+1}_n$ and $\upsilon_{b}(0)=\frac{6}{b+1} $, respectively.
\end{enumerate}

\end{theorem}

As mentioned earlier, our primary motivation behind Theorem \ref{ThmMain} is to obtain a deeper understanding of the previous results obtained in \cite{US}, where we studied the intermediate disorder regime on the diamond lattice when the environment variables are placed on the vertices of the graph rather than the edges. In that case we proved the analogues of Theorems \ref{ThmOld} and \ref{ThmOldII} in the $b = s$ case, namely that there is a similar phase transition in the asymptotic variance (with Gaussian fluctuations around it), although with a different value for the cutoff $\kappa_b$ and the intermediate disorder scaling of $\BA/n$ rather than $\BA/\sqrt{n}$. This differs from other polymer models subject to intermediate disorder scaling where the limiting variance is finite for all values of $\BA$ (including the $b < s$ vertex model of the diamond lattice, as we also proved in \cite{US}), which makes it interesting to probe the phase transition around the critical point of the variance blowup. We prefer to do this with the edge model as the corresponding recursion \eqref{Induct} is simpler and the resulting analysis is clearer.

Theorem~\ref{ThmMain} shows that an asymptotic variance analysis of the model still produces nontrivial results. The most interesting feature of Theorem \ref{ThmMain} is that the limiting results do not depend on the choice of the parameter $m$. Every time the parameter is increased it corresponds to probing the phase transition in the variance behavior beyond the previous critical value, but each and every time a new critical value lays ahead of it. It is in this sense that we regard the critical values as being nested. Moreover, the appearance of Gaussian fluctuations with the variance coefficient $\upsilon_b(\epsilon)$ that is independent of $m$ shows that the nature of the phase transition is always the same, although predicably the magnitude of the Gaussian fluctuations does decrease with $m$. 

We point out that, for several reasons, it is unlikely that the results of Theorem \ref{ThmMain} fall into the category of the intermediate disorder regime. First, the proper notion of strong and weak disorder is not as obvious for the edge-based model. For the vertex model the separation between the two is defined by the positivity of the martingale limit of the normalized partition functions, and it is known \cite{lacoin} that when $b \leq s$ all positive $\beta$ are in the strong disorder regime (i.e. the limit is almost surely zero for all positive $\beta$). In the edge model, however, there is no natural coupling of the partition function at different levels and hence no martingale limit. One could define the separation by replacing almost sure convergence with convergence in law, but an application of the environment tilting method of~\cite[Section 5]{lacoin} to the edge-based model yields only that $W_n(\beta)$ converges weakly to zero as $n \to \infty$ for sufficiently large values of $\beta$. It is plausible that $W_n(\beta)$ remains a positive quantity for $\beta$ small, in contrast to the site-based model. This is supported (or at least not contradicted) by simple combinatorics: the partition functions for both models are normalized sums of random variables label by the set of directed path $\Gamma_n$, however a pair of uniformly chosen paths share an expected number of $(b-1)n/b$ vertices but only one edge. For this reason the random variables in the sum defining the partition function for the vertex-based model are more correlated (on average) than those in the sum for the edge-based model, and hence the site-based model should reasonably be expected to be ``more disordered''.  This heurstic analysis seems to indicate that the intermediate disorder regime in the $b=s$ case is not accessed by scaling around the nested critical points of the variance blowup discussed in Theorem~\ref{ThmMain}.  


The bulk of this paper is dedicated to the proof of Theorem \ref{ThmMain}. In Section \ref{SecWeakDisorder} we analyze the asymptotic behavior of the variance and prove part (i) of Theorem \ref{ThmMain}. In Section \ref{SecLimitThm} we extend the analysis to prove the central limit theorem results of part (ii). Together these two sections contain all the necessary estimates and ideas to proofs of Theorems \ref{ThmOld} and \ref{ThmOldII}, which we leave to the reader. Finally, in Section \ref{SecQuenched} we describe how our results combined with those of Lacoin and Moreno \cite{lacoin} lead to a small improvement in known bounds between the quenched and annealed free energies; see Theorem \ref{ThmQuenched}.

\vspace{.1in}

{\bf Acknowledgments:} We thank an anonymous referee for several suggestions which led to a greatly improved article. Alberts gratefully acknowledges the support of Simons Foundation Collaboration Grant \#351687.

\section{Variance analysis}\label{SecWeakDisorder}

In this section we focus on controlling the variance of $W_{n}(\beta_{n, \epsilon}^{(m)}) $ in the respective cases $\epsilon < \eta_{b}$ (Section~\ref{SecVariance}) and $\epsilon> \eta_{b}$ (Section~\ref{SecExplode}). By the observation that $\beta_{n, \eta_{b}}^{(m)} = \beta_{n, 0}^{(m+1)}$ no generality is lost by assuming  $\epsilon \neq \eta_{b}$.  \vspace{.3cm}

Let $\varrho_{k}( \beta )\,:=\,  \operatorname{Var} (W_{k}(\beta)) $ and define
$M:\R^{+}\rightarrow \R^{+}$  as
$$\ds M(x)\, :=\, \frac{1}{b}\big[\big(1+ x \big)^b    \,-\,1   \big]\,.  $$
The distributional recursive equation~(\ref{Induct}) implies that $\varrho_{k+1}(\beta) =  M(\varrho_{k}(\beta))$, and thus $\varrho_k(\beta) = M^k\big(\varrho_0(\beta)\big)$, where $M^k$ denotes the $k$-fold composition of the map $M$.

\subsection{Variance convergence in the case $\epsilon<\eta_{b}$}\label{SecVariance}

The following lemma develops some technical results that we will need for the map $M:\R^{+}\rightarrow \R^{+}$.  The results in Lemma~\ref{LemInductionStep} are crafted for inductive use as the variance of $W_{k}\big( \beta_{n, \epsilon}^{(m)} \big) $  climbs through a hierarchy of scales as $k$ moves closer to $n$.

\begin{lemma}\label{LemInductionStep} Pick $\lambda_{c}>0$ and fix   $\epsilon \in (0,\eta_{b})$.    Let $X_{N}\in \mathbb{R}^+$  be  a  sequence such that   for $N\gg 1$
\begin{align}
 X_{N}  = \frac{ \kappa_{b}^2 }{\lambda_{c}} \frac{1}{N} +\mathit{O}\Big(\frac{1}{N^2}\Big)\, .\label{Assump}
\end{align}
\begin{enumerate}[i)]

\item As $N\rightarrow \infty$,
$$   M^{\lfloor \lambda_{c} N \rfloor }\big(X_{N}    \big)   \,  =\, \frac{\kappa_{b}^2}{\eta_{b}}\frac{1}{\log \big(\frac{N}{\log N} \big)}\,+\,\mathit{O}\bigg(\frac{1 }{|\log N|^2} \bigg) \, .$$

\item If  $\lambda \in (0,\lambda_{c})$, then as $N\rightarrow \infty$
$$  M^{\lfloor \lambda N \rfloor }\big(X_{N}  \big)   \,  = \, \frac{\kappa_{b}^2}{\lambda_{c}-\lambda}\frac{1}{N} \,+\,\mathit{O}\Big(\frac{1}{ N^2} \Big)\, .$$

\item  There is a  constant $C>0$ such that for all $N\in \mathbb{N}$
$$  \,C^{-1} \frac{1}{\log N} \,  \leq  \, M^{\lfloor \lambda_{c} N+\epsilon\log N \rfloor}\big(X_{N}\big)    \,  \leq  \,C \frac{1}{\log N } \, .$$

\item  Define $\alpha_{N}:=  \lfloor \lambda_{c} N+\epsilon \log N\rfloor $.   There is a $C>0$ such that for all $N\in \mathbb{N}$ and $k\leq \alpha_{N}$
$$ \displaystyle
 \frac{d}{dx} M^{\alpha_{N}-k}(x )\Big|_{x=M^{k}(X_{N}) }   \, \leq \,  C\Big(\frac{1}{\log N  }\Big)^2\frac{1}{\big|M^k(X_{N})  \big|^2}  \, . $$

\end{enumerate}

\end{lemma}

\begin{proof}Standard estimates show that it is enough to perform calculations with the third-order approximation
\begin{align}\label{DodoHead}
\widetilde{M}(x)\,=\,x+\frac{b-1}{2}x^2+  \frac{(b-1)(b-2)}{6}x^3\,,
\end{align}
of the polynomial $M(x)$, with the error terms being absorbed into the error terms and constants in the statement of the results. Using the quadratic approximation instead wouldn't change the basic form of the analysis, but it would incur a quantitative error when $b\neq 2$ in the coefficient $\kappa_{b}^2/\eta_{b}$ appearing in  part (i); this error is a consequence of neglecting the rightmost term in~\eqref{Selge} below.

\vspace{.3cm}

\noindent (i).  Define the sequence   $r_{j}^{(N)}\in [0,1)$ by
$$r_{j}^{(N)} \, := \,   1\,-\, \frac{\kappa_{b}^2}{\widetilde{M}^{j}(X_{N}  )} \frac{1}{N}   \,. $$
Note that by the assumption~(\ref{Assump})
\begin{align}\label{Inuit}
1\,-\,r_{0}^{(N)}\,=\, \lambda_{c} \,+\,\mathit{O}\Big(\frac{1}{N}\Big) \, .
\end{align}
It suffices for us to prove  that
\begin{align}\label{Lavae}
 1\,-\,r_{\lfloor  \lambda_{c} N \rfloor }^{(N)}\, = \,\eta_{b} \frac{\log N-\log\log N}{N}\, +\,\mathit{O}\Big( \frac{1}{N} \Big)  \, .  
 \end{align}
For  notational convenience we will identify $r_{j}\equiv r_{j}^{(N)}$ in the remainder of the proof.

The form of the map $\widetilde{M}$ and the definition of $\kappa_b$ implies that the $r_{j}$'s obey the recursive equation
\begin{align}\label{Selge}
  r_{j+1}     \,= &\,1\,-\,\bigg( \frac{1  }{ 1-  r_{j}  }\, +\,\frac{1}{N} \frac{1  }{ (1-  r_{j})^2  }\,+\,\frac{2(b-2)}{3(b-1)N^2} \frac{1  }{ (1-  r_{j})^3  } \bigg)^{-1}      \, .
\end{align}
Note that $r_{j+1}$ lies between  $r_{j}$ and $r_{j}+\frac{1}{N}$.   Applying Taylor's theorem to the function $f(x)=\frac{1}{1-x}$ at $x=r_{j}$ guarantees that there is a value $r_{j}^*$ in the interval $(r_{j},r_{j}+\frac{1}{N})$ such that
\begin{align}\label{SelgeII}
r_{j}+\frac{1}{N}  \,= &\, 1\,-\,\bigg(\frac{1  }{ 1-  r_{j}  }\, +\,\frac{1}{N} \frac{1  }{ (1-  r_{j})^2  }\, +\,\frac{1}{N^2} \frac{1  }{ (1-  r_{j}^{*})^3  } \bigg)^{-1}\, .
\end{align}
Notice that the expressions on the right sides of~(\ref{Selge}) and~(\ref{SelgeII}) differ only in the rightmost terms within the inverses. Define $\Delta_{j}\equiv \Delta_{j}^{(N)}$ as the difference 
\begin{align}
\Delta_{j} \,:=  \, \frac{1}{N^2} \frac{1  }{ (1-  r_{j}^{*})^3  }\,-\,\frac{2(b-2)}{3(b-1)N^2} \frac{1  }{ (1-  r_{j})^3  }\,.\label{PreHoink}
\end{align}
When $r_{j}$ is bounded away from $1$, then $\Delta_{j}$ is on the order of $\frac{1}{N^2}$.  When $r_{j}$ is close to $1$ with 
\begin{align}\label{Zeflter}
1-r_{j}\gg \frac{1}{N}\,,
\end{align}
in other terms, not ``too close," then we have the approximation  
\begin{align}
\Delta_{j} \, =\,  \eta_{b}   \frac{1  }{ N^2(1-  r_{j})^3  }\,+\,\mathit{O}\Big(\frac{1}{N^3(1-r_{j})^4    }    \Big) \label{Hoink}
\end{align}
since $1-\frac{2(b-2)}{3(b-1)}=\frac{b+1}{3(b-1)}=:\eta_{b}$. To be more clear, the relation~(\ref{Zeflter}) means that the index $j$ is restricted to a range $1\leq j  < u_{N}$ such that $\min_{1\leq j\leq u_{N}}N(1-r_{j})=      N(1-r_{u_{N}})\rightarrow \infty$ as  $N\rightarrow\infty$. 

 Looking at~(\ref{SelgeII}),  a second-order application of Taylor's theorem to the function $g(x)=1-\frac{1}{x}$ at the point $x=\frac{1}{1-r_{j+1}}$ yields that there is an $r_{j}^{**}\in \big(r_{j+1},r_{j}+\frac{1}{N}\big)$ such that
\begin{align}\label{Dimple}
r_{j}\,+\,\frac{1}{N}\, =\,r_{j+1}\,+\,\Delta_{j}(1-r_{j+1})^2      \,-\, \Delta_{j}^2 (1-r_{j}^{**})^3   \, .
\end{align}
When $1-r_{j}\gg \frac{1}{N}$, then~(\ref{Dimple}) and~(\ref{Hoink}) imply that the spacing between $r_{j+1}$ has $r_{j}$ has the form
\begin{align}\label{Hable}
r_{j+1}\,-\,r_{j}\,=\,\frac{1}{N}\,+\,\mathit{O}\Big( \frac{1}{N^2(1-r_{j}) } \Big)\,.
\end{align}

Fix some $0<\epsilon < \eta_b$ and define  $u_{N}\in \mathbb{N}$ as the smallest number $k=u_N$ such that $1-r_k  < \epsilon\frac{\log N}{N} $.   For $ 1\leq m\leq  \min(\lfloor\lambda_{c} N \rfloor,u_{N})$, the difference between $r_{m}$ and $1$ can be bounded using a telescoping sum combined with~(\ref{Inuit}) as follows:
\begin{align*}
1\,-\,r_{m}\, & = \,\lambda_{c}\,+\,(r_0-r_m)\,+\,\mathit{O}\Big(\frac{1}{N}\Big)     \,  \\
& = \,\frac{ \lfloor \lambda_{c} N \rfloor -m  }{ N }\,+   \,\sum_{j=0}^{m-1}\Big(r_{j}+\frac{1}{N}-r_{j+1}   \Big)\,+\, \mathit{O}\Big(\frac{1}{N}\Big) \, \\
 & =\,\frac{ \lfloor \lambda_{c} N \rfloor -m  }{ N }\,+   \,\sum_{j=0}^{m-1}\Delta_{j}(1-r_{j+1})^2  \, - \,\sum_{j=0}^{m-1} \Delta_{j}^2 (1-r_{j}^{**})^3\,+\, \mathit{O}\Big(\frac{1}{N}\Big) \, \\
 & =\,\frac{ \lfloor \lambda_{c} N \rfloor -m  }{ N }\,+   \,\frac{\eta_{b}}{N^2}\bigg(1\,+\,\mathit{O}\Big( \frac{1}{\log N  }  \Big)   \bigg)\sum_{j=0}^{m-1} \frac{1}{1-r_{j}}\,+\,\mathit{O}\Big(\frac{1}{N}\Big)\,.
\end{align*} 
The third equality follows from \eqref{Dimple} and the fourth from \eqref{Hoink}. Since the $r_{j}$'s are spaced apart by $\frac{1}{N}+\mathit{O}\big(\frac{1}{N\log N}   \big) $ by~(\ref{Hable}), we have the Riemann approximation $\frac{1}{N}\sum_{j=0}^{m-1} \frac{1}{1-r_{j}}=\big(1\,+\,\mathit{O}\big( \frac{1}{\log N}  \big)   \big)\int_{r_0}^{r_{m}}\frac{1}{1-x}dx$, yielding that
\begin{align}
 1\,-\,r_{m}\, &=\,\frac{ \lfloor \lambda_{c} N \rfloor -m  }{ N }\,+   \,\frac{\eta_{b}}{N}\bigg(1\,+\,\mathit{O}\Big( \frac{1}{\log N}  \Big)   \bigg)\int_{r_0}^{r_{m}}\frac{1}{1-x}dx \,+\,\mathit{O}\Big(\frac{1}{N}\Big)\nonumber  \\
 & =\,\frac{ \lfloor \lambda_{c} N \rfloor -m  }{ N }\,+ \,\frac{\eta_{b}}{N}\bigg(1\,+\,\mathit{O}\Big( \frac{1}{\log N}  \Big)   \bigg)\log\bigg(\frac{1-r_0}{1-  r_{m} }\bigg)\,+\,\mathit{O}\Big(\frac{1}{N}\Big)\,\nonumber \\
 & =\,\frac{ \lfloor \lambda_{c} N \rfloor -m  }{ N }\,+ \,\frac{\eta_{b}}{N}\log\Big(\frac{1}{1-  r_{m} }\Big)\,+\,\mathit{O}\Big(\frac{1}{N}\Big)\, .\label{Naple}
\end{align}

We would like to apply  equality~(\ref{Naple}) with $m=\lfloor\lambda_{c} N \rfloor$, but we first need to verify that  $ \lfloor\lambda_{c} N \rfloor$ is smaller than $u_{N}$  when $N$ is sufficiently large.  Equation~(\ref{Naple}) implies that when $m\leq \min(\lfloor \lambda_{c} N \rfloor , u_N)$
\begin{align}\label{T}
  1\,-\,r_m\, \geq   \, \frac{\eta_{b}}{N}\log\Big(\frac{1}{1-  r_{m} }\Big)\,+\,\mathit{O}\Big(\frac{1}{N}\Big)      \,.
\end{align}
Suppose to reach a contradiction that $u_{N}\leq  \lfloor\lambda_{c} N \rfloor$ and $N\gg1$.  Then~(\ref{T}) combined with the definition of  
$u_{N}$ implies that for $m=u_{N}$
\begin{align*}
 1\,-\,r_m\, \geq  & \, \frac{\eta_{b}}{N}\log\Big(\frac{N}{\epsilon \log N}\Big)\,+\,\mathit{O}\Big(\frac{1}{N}\Big)   \\   
   = & \, \frac{\eta_{b}\log N}{N}\,-\,\frac{\eta_{b}\log \log N}{N}\,+\,\mathit{O}\Big(\frac{1}{N}\Big)  \,.    
  \end{align*}
However, as $N\rightarrow \infty$  the above contradicts that  $1-r_m <\frac{\epsilon}{N}\log N  $ for $\epsilon<\eta_b$, which holds by definition of $m=u_N$.  Therefore, $u_{N}> \lfloor\lambda_{c} N \rfloor$  when $N$ is large enough.

 Since $u_{N}> \lfloor\lambda_{c} N \rfloor$ holds for  $N\gg 1$, we can plug $m=\lfloor\lambda_{c} N \rfloor$ in to~(\ref{Naple}) to get $1-r_{\lfloor\lambda_{c} N \rfloor}= \frac{\eta_{b}}{N}\log\big(\frac{1}{1-  r_{\lfloor\lambda_{c} N \rfloor} }\big)+\mathit{O}\big(\frac{1}{N}\big)$, and thus
\begin{align*}
 1-r_{\lfloor\lambda_{c} N \rfloor} &\,=\, \frac{\eta_{b}}{N}\log\bigg(\frac{1}{\frac{\eta_{b}}{N}\log\big(\frac{1}{1-  r_{\lfloor\lambda_{c} N \rfloor} }\big)+\mathit{O}\big(\frac{1}{N}\big)}\bigg)\,+\,\mathit{O}\Big(\frac{1}{N}\Big)\\ & \,=\,  \frac{\eta_{b}\log N}{N}\,-\,\frac{\eta_{b}\log\log\big( \frac{1}{1-r_{\lfloor\lambda_{c} N \rfloor}}  \big)}{N} \,+\,\mathit{O}\Big(\frac{1}{N}\Big)  
 \\ & \,=\,  \frac{\eta_{b}\log N}{N}\,-\,\frac{\eta_{b}\log\log\Big( \frac{1}{\frac{\eta_{b}\log N}{N}+o\big(  \frac{\log N}{N} \big)  }  \Big)}{N} \,+\,\mathit{O}\Big(\frac{1}{N}\Big)  \,. 
 \end{align*}
The above implies that $1-r_{\lfloor \lambda_{c} N \rfloor}=\eta_{b}\frac{\log N -\log\log N}{N}+\mathit{O}(\frac{1}{N})$.

\vspace{.4cm}

\noindent (ii)   Let $ r_{j}^{(N)}\equiv r_{j}$ be defined as in part (i).  The result follows by showing that
\begin{align*}
 1\,-\,r_{\lfloor \lambda N\rfloor  }  \, = \,\lambda_{c}\,-\,\lambda\, +\,\mathit{O}\Big( \frac{1}{N} \Big)  \,.
\end{align*}
By \eqref{Hoink} and \eqref{Dimple} there is a $C>0$ such that for all $N\in \mathbb{N}$ and $j\leq  \lambda N$
$$  \Big| r_{j+1}\,-\,r_{j}\,-\frac{1}{N}\Big|\,\leq\,\frac{C}{N^2}  \,.   $$
The above combined with~(\ref{Inuit}) implies the result.

\vspace{.4cm}

\noindent (iii)  It is equivalent to prove the result with $\lfloor \lambda_c N+\epsilon\log(N)\rfloor$ replaced by $\lfloor \lambda_c N+\epsilon\log(\frac{N}{\log N})\rfloor$ for all $\epsilon\in (0,\eta_{b})$.  Applying (i) in the second and equality below gives us
\begin{align*}
M^{\lfloor \lambda_c N+\epsilon\log(\frac{N}{\log N})\rfloor}(  X_N )   \,=\,& M^{\lfloor \lambda_{c} N+\epsilon\log( \frac{N}{\log N} ) \rfloor -\lfloor \lambda_{c} N\rfloor }   \Big(M^{\lfloor \lambda_{c} N \rfloor}\big(X_{N}\big) \Big)\\  \,=\,& M^{\lfloor \lambda_{c} N+\epsilon\log( \frac{N}{\log N} )\rfloor -\lfloor \lambda_{c} N\rfloor }  \bigg( \frac{\kappa_{b}^2}{\eta_{b}} \frac{1}{\log (\frac{N}{\log N} )  }+\mathit{O}\Big(  \frac{1}{|\log N |^2}  \Big)  \bigg)
\\  \,=\,&\frac{\kappa_{b}^2}{\eta_{b}-\epsilon } \frac{1}{\log \big(\frac{N}{\log N} \big)  }+\mathit{O}\Big(  \frac{1}{|\log N |^2}  \Big) \,.
\end{align*}
The third equality follows from (ii) since $\frac{1}{|\log N|^2}= \mathit{O}\Big(\frac{1}{|\log(\frac{N}{\log N})|^2}\Big)$.

\vspace{.4cm}

\noindent (iv) Let $r_{j}$ be defined as in part (i).  The chain rule and the definition of $\widetilde{M}$ give:
\begin{align*}
 \frac{d}{dx}\widetilde{M}^{\alpha_{N}-k}(x)  \Big|_{x =\widetilde{M}^{k}(X_{N})   } \, = &\, \prod_{j=1}^{\alpha_{N}-k} \widetilde{M}'\Big( \widetilde{M}^{k+j-1}(X_{N})\Big)\, \\
 \, = & \, \prod_{j=1}^{\alpha_{N}-k}  \bigg(1+ (b-1)\widetilde{M}^{k+j-1}(X_{N})+ \frac{(b-1)(b-2)}{2}\Big(\widetilde{M}^{k+j-1}(X_{N})\Big)^2\bigg)\,.
\end{align*}
 The terms $\widetilde{M}^{k+j-1}(X_{N})$ are bounded by $\widetilde{M}^{\alpha_{N}}(X_{N})$, which is $  \mathit{O}\big( 1/\log N\big)$ by part (iii), since $\widetilde{M}^{m}(X_{N})$ increases with $m$. 
\begin{align*}
\frac{d}{dx}\widetilde{M}^{\alpha_{N}-k}(x)  \Big|_{x =\widetilde{M}^{k}(X_{N})} \leq  &\,\exp\Bigg\{ (b-1)\Big(1+\frac{c}{\log N }   \Big)\sum_{j=1}^{\alpha_{N}-k}\widetilde{M}^{k+j-1}(X_{N})  \Bigg\}\, .
\intertext{By definition of  $r_j$  for $k\leq j < \lfloor \lambda_{c} N\rfloor$, we can write  }
   = & \,\exp\Bigg\{   \Big(1+\frac{c}{\log N}   \Big)\bigg( \frac{2}{N}\sum_{j=k}^{\lfloor \lambda_{c} N\rfloor-1}\frac{1}{ 1-  r_{j}} \,+\, (b-1)\sum_{j=\lfloor \lambda_{c} N\rfloor}^{\alpha_{N}-1}\widetilde{M}^{j}(X_{N})  \bigg)  \Bigg\}\,
\, .
\end{align*}
The partial sum $ \sum_{j=\lfloor \lambda_{c} N\rfloor}^{\alpha_{N}-1}\widetilde{M}^{j}(X_{N})   $ is uniformly bounded by a constant as a consequence of part (iii) since
$$\sum_{j=\lfloor \lambda_{c} N\rfloor}^{\alpha_{N}-1}\widetilde{M}^{j}(X_{N})   \,\leq \, \epsilon \log N \sup_{1\leq j\leq \alpha_{N}}\widetilde{M}^{j}(X_{N})  \,=\,\mathit{O}(1)\,.    $$
Moreover, by using Riemann sum approximations similar to those in part (i), we can see that the difference between $\frac{1}{N}\sum_{j=k}^{ \lfloor \lambda_{c} N\rfloor-1}\frac{1}{ 1-  r_{j}}  $ and  $\int_{r_{k}}^{r_{ \lfloor \lambda_{c} N\rfloor}}\frac{1}{1-s}ds$ is uniformly bounded by a constant: 
$$ \sum_{j=k}^{ \lfloor \lambda_{c} N\rfloor-1}\frac{1}{ 1-  r_{j}}\,-\,\int_{r_{k}}^{r_{ \lfloor \lambda_{c} N\rfloor}}\frac{1}{1-s}ds
\,=\,  \mathit{O}\Bigg(\frac{1}{\log N}\int_{r_{k}}^{r_{ \lfloor \lambda_{c} N\rfloor}}\frac{1}{1-s}ds\Bigg) \,=\,\mathit{O}\Bigg( \frac{\log\big(\frac{1}{1-r_{ \lfloor \lambda_{c} N\rfloor} } \big)}{\log N}  \Bigg)\,=\,\mathit{O}(1)\,. $$
The first equality above follows since the spacing between the $r_j$'s is $\frac{1}{N}+\mathit{O}\big(\frac{1}{N^2(1-r_{j})}  \big)=\frac{1}{N}+\mathit{O}\big(\frac{1}{N\log N}  \big)$ by~(\ref{Hable}), and the third equality is by~(\ref{Lavae})   Hence there is a  $C>0$ such that   
\begin{align}
\frac{d}{dx}\widetilde{M}^{\alpha_{N}-k}(x)  \Big|_{x =\widetilde{M}^{k}(X_{N})}\leq & \,C\exp\bigg\{   2 \Big(1+\frac{c}{\log N}   \Big) \int_{r_{k}}^{r_{ \lfloor \lambda_{c} N\rfloor}} \frac{1}{1-s}ds    \bigg\}\nonumber  \\
  = & \,C\exp\bigg\{  2  \Big(1+\frac{c}{\log N }   \Big)\Big(\log( 1-r_{k} ) -\log( 1-r_{ \lfloor \lambda_{c} N\rfloor} )    \Big)       \bigg\}\nonumber \\
 \leq & \,C'\frac{ ( 1-r_{k} )^2 }{ ( 1-r_{ \lfloor \lambda_{c} N\rfloor} )^2 }   \nonumber \\
\leq &\,C''\Big(\frac{1}{\log N }\Big)^2\frac{1}{\big(\widetilde{M}^k(X_{N})  \big)^2 }
\, .\label{Jakdah}
\end{align}
In the last inequality, we have also used that $  (1- r_{k}  )^2    =   \kappa_{b}^{2} \big(\widetilde{M}^k(X_{N})N \big)^{-2}  $
and that  the factor $(1-r_{\lfloor \lambda_{c} N\rfloor})^{-1}$  is bounded by a constant multiple of $N/\log N$ by the analysis in the proof of part (i).

\end{proof}

\begin{remark} Recall that $\ell_{x}:=\log\big(\frac{1+x}{\log(1+x)}\big)$.  The results i, iii, and iv of Lemma~\ref{LemInductionStep} can be equivalently stated in terms of $\ell_x$ as follows:

\begin{enumerate}

\item[i)] As $N\rightarrow \infty$,
$$   M^{\lfloor \lambda_{c} N \rfloor }\big(X_{N}    \big)   \,  =\, \frac{\kappa_{b}^2}{\eta_{b}}\frac{1}{\ell_N}\,+\,\mathit{O}\Big(\frac{1 }{|\ell_N|^2} \Big) \, .$$

\item[iii)]  There is a  constant $C>0$ such that for all $N\in \mathbb{N}$
$$  \,C^{-1} \frac{1}{\ell_N} \,  \leq  \, M^{\lfloor \lambda_{c} N+\epsilon\ell_N \rfloor}\big(X_{N}\big)    \,  \leq  \,C \frac{1}{\ell_N } \, .$$

\item[iv)]  Define $\alpha_{N}:=  \lfloor \lambda_{c} N+\epsilon \ell_N\rfloor $.   There is a $C>0$ such that for all $N\in \mathbb{N}$ and $k\leq \alpha_{N}$
$$ \displaystyle
 \frac{d}{dx} M^{\alpha_{N}-k}(x )\Big|_{x=M^{k}(X_{N}) }   \, \leq \,  C\Big(\frac{1}{\ell_N  }\Big)^2\frac{1}{\big|M^k(X_{N})  \big|^2}  \, . $$

\end{enumerate}

\end{remark}

The following lemma states that when $\epsilon < \eta_b$ the variance of $ \sqrt{ \ell^{m}_n}W_{n}\big( \beta_{n, \epsilon}^{(m)}\big) $  converges to $\upsilon_{b}(\epsilon)$ as $n \to \infty$, which  we will need to prove the central limit theorem  in part (ii) of Theorem~\ref{ThmMain}. This convergence implies, in particular, part (i) of Theorem~\ref{ThmMain} in the $\epsilon< \eta_{b}$ case.

\begin{lemma}\label{LemVarConv} Fix $m\in \mathbb{N}$ and $\epsilon< \eta_{b}$.  Then
\[   \lim_{n \to \infty}  \ell^{m}_{n} \varrho_{n}\big( \beta_{n, \epsilon}^{(m)} \big)  = \upsilon_{b}(\epsilon) = \frac{\kappa_b^2}{\eta_b - \epsilon}\,.   \]
\end{lemma}

\begin{proof} Recall that $k\mapsto \varrho_{k}\big(\beta_{n,\epsilon}^{(m)}\big)$ satisfies
the recursive relation $ \varrho_{k+1}\big(\beta_{n,\epsilon}^{(m)}\big) = M\big( \varrho_{k}(\beta_{n,\epsilon}^{(m)}) \big) $, and that the asymptotics of  $\varrho_{0}\big(\beta_{n,\epsilon}^{(m)}\big)=\textup{Var}\big(W_{0}\big( \beta_{n,\epsilon}^{(m)}  \big)\big)$ have the form~(\ref{varrho}).  Define $N_{n}^{(k)}\in \mathbb{N}$ for $k\leq m$  and   $ N_{n,\epsilon}^{(m)}\in \mathbb{N}$ as
$$ N_{n}^{(k)}\,:=\,  n  \, +\,  \left\lfloor  \eta_{b} \sum_{j=1}^{k-1}\ell^{j}_n  \right\rfloor \hspace{.8cm} \text{and} \hspace{.8cm} N_{n,\epsilon}^{(m)}\,:=\,  n  \, +\, \left\lfloor    \eta_{b} \sum_{j=1}^{m-1}\ell^{j}_n \,+\,\epsilon \ell^{m}_n \right\rfloor  .     $$
Define $\beta_{n}:=\beta_{n,0}^{(1)}= \kappa_{b}/\sqrt{n}-\tau\kappa_{b}^2/2n$.   The convergence of $\ell^{m}_{n} \varrho_{n}\big( \beta_{n, \epsilon}^{(m)} \big)$ to $\upsilon_{b}(\epsilon)$ is implied by the following statements, which are proved below.
\begin{enumerate}[I)]

\item  As $n \to \infty$
$$      \varrho_{N_{n, \epsilon}^{(m)}}(\beta_{n})   \, =\,\frac{  \kappa_{b}^2  }{\eta_{b}-  \epsilon }\frac{1}{\ell^{m}_n}+\mathit{O}\bigg( \frac{1}{(\ell^{m}_n)^2}\bigg)\, .    $$

\item   The difference between $\varrho_{n}\big(\beta_{n,\epsilon}^{(m)}\big)$ and $\varrho_{N_{n, \epsilon}^{(m)}}(\beta_{n})      $ has the bound
$$\Big|\varrho_{n}\big(\beta_{n,\epsilon}^{(m)}\big) \, -  \,  \varrho_{N_{n, \epsilon}^{(m)}}(\beta_{n})     \Big| \, \leq \,  \frac{C}{ ( \ell^{m}_n)^2 }\,  $$
for some $C>0$ and all $n>0$.

\end{enumerate}

\vspace{.3cm}

We prove (\textup{I}) by analyzing  $$ \varrho_{N_{n, \epsilon}^{(m)}}(\beta_{n}) \,=\, M^{N_{n, \epsilon}^{(m)} }\big( \rho_{0}(\beta_{n}) \big)=   M^{N_{n, \epsilon}^{(m)} }\big(\kappa_{b}^2/ n+\mathit{O}(n^{-2})\big)$$
 through an induction argument using Lemma~\ref{LemInductionStep}.  Note that the following asymptotic formula holds  for $r=1$ and $n\gg 1$  by an application of part (i) of  Lemma~\ref{LemInductionStep}
\begin{align}\label{DoInd}
\varrho_{N_{n}^{(r)}}(\beta_{n}) \,=\,M^{N_{n}^{(r)}}\bigg(\frac{\kappa_{b}^2}{n}+\mathit{O}\Big(\frac{1}{n^{2}}\Big)\bigg)\, = \, \frac{\kappa_{b}^2}{\eta_{b}}\frac{1}{\ell^{r}_n}+\mathit{O}\bigg( \frac{1}{(\ell^{r}_n)^2}\bigg)\,.
\end{align}
   Moreover, if~(\ref{DoInd}) holds for some $1\leq r< m$, then it also must hold for $r+1$ since
\begin{align*}
M^{N_{n}^{(r+1)}}\bigg(\frac{\kappa_{b}^2}{n}+\mathit{O}\Big(\frac{1}{n^{2}}\Big)\bigg)\, =&  \,M^{N_{n}^{(r+1)}-N_{n}^{(r)}}\Bigg(M^{N_{n}^{(r)}}\bigg(\frac{\kappa_{b}^2}{n}+\mathit{O}\Big(\frac{1}{n^{2}}\Big)\bigg)\Bigg)\\ = &M^{N_{n}^{(r+1)}-N_{n}^{(r)}}\Bigg( \frac{\kappa_{b}^2}{\eta_{b}}\frac{1}{\ell^{r}_n}+\mathit{O}\bigg( \frac{1}{(\ell^{r}_n)^2}\bigg)\Bigg) \, ,
\intertext{and by part (i) of  Lemma~\ref{LemInductionStep} the above is equal to}
  = &\frac{\kappa_{b}^2}{\eta_{b}}\frac{1}{\ell^{r+1}_n}+\mathit{O}\bigg( \frac{1}{ (\ell^{r+1}_n)^2}\bigg)\, .
\end{align*}
Therefore,  (\ref{DoInd}) holds for all $1\leq r\leq m$.   We can apply part (ii) of  Lemma~\ref{LemInductionStep}  and the same reasoning as above to conclude that
\begin{align*}
M^{N_{n,\epsilon}^{(m)}}\bigg(\frac{\kappa_{b}^2}{n}+\mathit{O}\Big(\frac{1}{n^{2}}\Big)\bigg)\, =\,M^{N_{n,\epsilon}^{(m)}-N_{n}^{(m)}}\Bigg(M^{N_{n}^{(m)}}\bigg(\frac{\kappa_{b}^2}{n}+\mathit{O}\Big(\frac{1}{n^{2}}\Big)\bigg)\Bigg)\, =\,\frac{\kappa_{b}^2}{\eta_{b}-\epsilon  }\frac{1}{\ell^{m}_n}\,+\,\mathit{O}\bigg( \frac{1}{(\ell^{m}_n)^2}\bigg)\, .
\end{align*}

\vspace{.4cm}

For (\textup{II}) notice that the terms $\varrho_{n}\big(\beta_{n,\epsilon}^{(m)}\big) $ and $\varrho_{N_{n, \epsilon}^{(m)}}(\beta_{n})$ can be written as
\begin{align}\label{FirRho}
\varrho_{n}\big(\beta_{n,\epsilon}^{(m)}\big) \,= \,M^{n}\big( \varrho_{0}\big(\beta_{n,\epsilon}^{(m)}\big)\big)\,, \quad 
 \varrho_{N_{n, \epsilon}^{(m)}}(\beta_{n})    \,=\, M^{n}\big(\varrho_{N_{n, \epsilon}^{(m)}-n}(\beta_{n})\big)     \,.
\end{align}
 Moreover, it can be shown that 
\begin{align}\label{Frunkle}
\Big| \varrho_{0}\big(\beta_{n,\epsilon}^{(m)}\big) \,-\, \varrho_{N_{n, \epsilon}^{(m)}-n}(\beta_{n})      \Big| \,\leq \,      \frac{C }{n^2}
\end{align}
for some $C>0$.  To see~(\ref{Frunkle}) first recall that by~(\ref{varrho})
\begin{align}\label{Boo}
\varrho_{0}\big(\beta_{n,\epsilon}^{(m)}\big)\,=\,\frac{\kappa_{b}^2}{n}\,+\, \frac{\kappa_{b}^2   }{ n^2  } \big( N_{n, \epsilon}^{(m)}-n\big)\,+\,\mathit{O}\Big(\frac{1}{n^{2}}\Big)\,.
\end{align}
Secondly, a linearization  of the map $M(x)$ around $x=\kappa_{b}^2/n$ yields that
\begin{align*}
\bigg|M\Big(\frac{\kappa_{b}^2}{n}+\Delta x\Big)\, -\,\Big(\frac{\kappa_{b}^2}{n}\,+\,\Delta x\Big)\,-\,\frac{\kappa_{b}^2}{n^2} \bigg|  \,\leq \, \frac{c}{n^{\frac{5}{2}}}
\end{align*}
for some $c>0$ and all $\Delta x $ in the range $[0,n^{-\frac{3}{2}}]$.  If $u_{n}\in \mathbb{N}$ is  the  first value $j=u_{n}$ such that $ \varrho_{j}(\beta_{n})=  M^{j}\big(\varrho_{0}(\beta_{n})\big)>n^{-\frac{3}{2}}$, then for $1\leq k\leq u_{n}$ a telescoping sum gives us
\begin{align}
\bigg|\varrho_{k}(\beta_{n})\,-\, \frac{\kappa_{b}^2}{n} \, -\, \frac{\kappa_{b}^2}{n^2} k\bigg| \,\leq \,&\Big| \varrho_{0}(\beta_{n})  -\frac{\kappa_{b}^2}{n} \Big|\,+\,\sum_{j=1}^{k  }\bigg|  M^{j}\big(\varrho_{0}(\beta_{n})\big) \,-\,  M^{j-1}\big(\varrho_{0}(\beta_{n})\big) \,-\, \frac{\kappa_{b}^2}{n^2}  \bigg| \nonumber \\
\,= \,&\mathit{O}\Big( \frac{1}{n^2}  \Big)\,+\,\sum_{j=1}^{k  }\bigg|  M\big(M^{j-1}\big(\varrho_{0}(\beta_{n})\big)\big) \,-\,  M^{j-1}\big(\varrho_{0}(\beta_{n})\big) \,-\, \frac{\kappa_{b}^2}{n^2}  \bigg| \nonumber \\  \,\leq \, &\mathit{O}\Big( \frac{1}{n^2}  \Big)\,+\, \frac{c}{n^{\frac{5}{2}}} k\,.\label{PreDimpere}
\end{align}
From~(\ref{PreDimpere}) we can see that $u_n> N_{n, \epsilon}^{(m)}-n$ for large enough $n$, and thus
\begin{align}
\bigg|\varrho_{N_{n, \epsilon}^{(m)}} ( \beta_{n})    \,-\, \frac{\kappa_{b}^2}{n} \, -\, \frac{\kappa_{b}^2}{n^2}  \big( N_{n, \epsilon}^{(m)}-n\big)\bigg| \,\leq \,\mathit{O}\Big( \frac{1}{n^2}  \Big)\,+\,\frac{c}{n^{\frac{5}{2}}}\big(N_{n, \epsilon}^{(m)}-n  \big)\,= \, \mathit{O}\Big( \frac{1}{n^2}  \Big)  \,.\label{Dimpere}
\end{align}
 Combining~(\ref{Boo}) with~(\ref{Dimpere}) implies~(\ref{Frunkle}). Now by \eqref{FirRho} we have the equality
\begin{align*}
\bigg|\varrho_{n}\big(\beta_{n,\epsilon}^{(m)}\big) \, -  \,  \varrho_{N_{n, \epsilon}^{(m)}} ( \beta_{n})     \bigg| \, =\,&\bigg|M^{n}\Big( \varrho_{0}\big(\beta_{n,\epsilon}^{(m)}\big) \Big) - M^{n}\big(\varrho_{N_{n, \epsilon}^{(m)}-n}(\beta_{n})\big)     \bigg| \\
 \leq \,  &  \Big|  \varrho_{0}\big(\beta_{n,\epsilon}^{(m)}\big)  \,-\, M^{N_{n, \epsilon}^{(m)}-n}\big(\varrho_{0}(\beta_{n})\big)   \Big|\,  \frac{d}{dx}M^{n}(x)\Big|_{x=\textup{max}\big( \varrho_{0}\big(\beta_{n,\epsilon}^{(m)}\big) ,\, \varrho_{N_{n, \epsilon}^{(m)}-n}(\beta_{n})\big)  }\,.
\end{align*}
The inequality uses that the derivative of $M^{n}(x)$ is increasing.  The first term above can be bounded by~(\ref{Frunkle}).   By choosing any $\widehat{\epsilon}\in (\epsilon, \eta_{b})$, the term $ \varrho_{0}\big(\beta_{n,\epsilon}^{(m)}\big) $ will be smaller than $\varrho_{N_{n, \epsilon}^{(m)}-n}(\beta_{n}) $ for large enough $n$ as a consequence of~\eqref{Frunkle}, so
\begin{align*}
\big|\varrho_{n}\big(\beta_{n,\epsilon}^{(m)}\big) \, -  \,  \varrho_{N_{n, \epsilon}^{(m)}} ( \beta_{n})     \big| \leq \,  &  \frac{C}{n^2}\,   \frac{d}{dx}M^{n}(x)\Big|_{x= \varrho_{N_{n, \epsilon}^{(m)}-n}(\beta_{n}) }\,.
\end{align*}
To complete the proof of (II) we need to show that the derivative above is bounded by a constant multiple of $n^2/(\ell_{n}^{m})^2$ by writing $M^{n}(x)$ as
\begin{align*}
M^{n}(x)\,=\,M^{N_{n,\widehat{\epsilon}}^{(m)}-N_{n}^{(m-1)}   }\circ \cdots  \circ M^{N_{n}^{(2)} -N_{n}^{(1)}}\circ M^{N_{n}^{(1)}-(N_{n,\widehat{\epsilon}}^{(m)}-n)}(x   )\,.
\end{align*}
The chain rule gives us
\begin{align*}
\frac{d}{dx}M^{n}(x)&\Big|_{x=\varrho_{N_{n, \epsilon}^{(m)}-n}(\beta_{n}) }\\  \,=\,&\frac{d}{dx}M^{N_{n,\widehat{\epsilon}}^{(m)}-N_{n}^{(m-1)}   }(x)\Big|_{x= \varrho_{N_{n}^{(m-1)}}(\beta_{n}) } \frac{d}{dx}M^{N_{n}^{(m-1)}-N_{n}^{(m-2)}   }(x)\Big|_{x= \varrho_{N_{n}^{(m-2)}}(\beta_{n}) } \\ & \cdots  \frac{d}{dx}M^{N_{n}^{(2)}-N_{n}^{(1)}   }(x)\Big|_{x= \varrho_{N_{n}^{(1)}}(\beta_{n}) }\frac{d}{dx} M^{N_{n}^{(1)}-(N_{n,\widehat{\epsilon}}^{(m)}-n)}(x   )\Big|_{x= \varrho_{N_{n, \epsilon}^{(m)}-n}(\beta_{n}) }\,.
\end{align*}
With the asymptotics~(\ref{DoInd}) for  $\varrho_{N_{n}^{(r)}}(\beta_{n})$ in hand, we can apply (iv) of Lemma~\ref{LemInductionStep} to the $m$ derivatives above to get that
$$ \frac{d}{dx}M^{n}(x)\Big|_{x= M^{N_{n, \widehat{\epsilon}}^{(m)}-n}\big(\varrho_{0}(\beta_{n})\big) }\,\leq \,C^m \Big(\frac{1}{\ell_n^m  }\Big)^2\frac{1}{|\varrho_{N_{n}^{(m-1)}(\beta_{n})}  |^2} \cdots \Big(\frac{1}{\ell_n^{2}  }\Big)^2\frac{1}{|\varrho_{N_{n}^{(1)}}(\beta_{n})  |^2}\Big(\frac{1}{\ell_n  }\Big)^2\frac{1}{|\varrho_{N_{n, \widehat{\epsilon}}^{(m)}-n}(\beta_{n})  |^2}\,. $$
The right side above contracts to a multiple of $n^2/(\ell_{n}^{m})^2$ through a telescoping product.

\end{proof}

\subsection{Variance explosion when $\epsilon>\eta_{b}$}\label{SecExplode}

In this section we assume $\epsilon > \eta_b$ and define
\[ \alpha_{n,\epsilon}^{(m)}:=\big\lfloor \ell^{(m)}_n\, (\epsilon -\eta_{b})\big\rfloor. \]
The following lemma is a straightforward application of Lemma~\ref{LemVarConv}.

\begin{lemma}\label{LemCopy} Fix $m\in \mathbb{N}$ and $\epsilon> \eta_{b}$. Then
\begin{align*}
  \lim_{n \to \infty} \ell^{m+1}_n\varrho_{n-\alpha_{n,\epsilon}^{(m)}}\big( \beta_{n, \epsilon}^{(m)} \big) = \frac{6}{b+1}  \,.
\end{align*}

\end{lemma}

\begin{proof}This follows from Lemma~\ref{LemVarConv} by replacing the system size, $n$, by $n-\alpha_{n,\epsilon}^{(m)}$ since $\ell^{m+1}_n \approx \ell^{m+1}_{n-\alpha_{n,\epsilon}^{(m)}}$ and
\begin{align*}
\beta_{n, \epsilon}^{(m)}\,:=\,&\frac{ \kappa_{b} }{\sqrt{n}} \,-  \,\frac{ \tau\kappa_{b}^2 }{2n}  \,+\, \frac{ \kappa_{b}}{2n^{\frac{3}{2}}}\bigg( \eta_{b} \sum_{k=1}^{m-1}   \ell^{k}_{n} \,+\, \epsilon \ell^{m}_{n}\bigg)\\ \, =\,&\frac{ \kappa_{b} }{(n-\alpha_{n,\epsilon}^{(m)})^{\frac{1}{2}}  } \,-  \,\frac{ \tau\kappa_{b}^2 }{2(n-\alpha_{n,\epsilon}^{(m)})}     \,+\, \frac{ \eta_{b} \kappa_{b}}{2\big(n-\alpha_{n,\epsilon}^{(m)}\big)^{\frac{3}{2}}} \sum_{k=1}^{m}    \ell^{k}_{n-\alpha_{n,\epsilon}^{(m)}}\,+\,\mathit{O}\Big( \frac{1}{n^{\frac{3}{2}}}  \Big)    \,,
\end{align*}
where we have used that $\beta_{n, \eta_{b}}^{(m)}=\beta_{n, 0}^{(m+1)}$.

\end{proof}

\begin{proof}[Proof of part (i) of Theorem~\ref{ThmMain} in the $\epsilon>\eta_{b}$ case]  Let $\widehat{M}(x) := x + \frac{b-1}{2} x^2$ and noticing that $M(x) \geq \widehat{M}(x)$ for $x \geq 0$ we have
 \begin{align}
\textup{Var}\big( W_{n}\big( \beta_{n, \epsilon}^{(m)}  \big)\big)\, =\, \varrho_ {n }\big(\beta_{n, \epsilon}^{(m)}\big)    \, =\, M^{\alpha_{n,\epsilon}^{(m)}}\Big(\varrho_ {n-\alpha_{n,\epsilon}^{(m)} }\big(\beta_{n, \epsilon}^{(m)}\big)\Big)\, \geq \,& \widehat{M}^{\alpha_{n,\epsilon}^{(m)}}\Big( \varrho_ {n-\alpha_{n,\epsilon}^{(m)} }\big(\beta_{n, \epsilon}^{(m)}\big)   \Big) \notag \\
\geq \,&\widehat{M}^{\alpha_{n,\epsilon}^{(m)}}\Big(\frac{ L  }{  \ell^{m}_n  }\Big)\,. \label{Dimp}
\end{align}
The last inequality holds for any fixed $L>0$ and large enough $n$  since $\varrho_{n-\alpha_{n,\epsilon}^{(m)} }\big(\beta_{n, \epsilon}^{(m)}\big)  \propto  1/\ell^{m+1}_n $ by Lemma~\ref{LemCopy}.  It will be convenient for us to choose  some $L$ greater than $\frac{16\epsilon}{(\epsilon -\eta_{b})(b-1)}$.
Writing $\widehat{M}(x)$ as $x\big(1+\frac{b-1}{2}x\big) $, it is clear that for any  $K>0$ and $r\in \mathbb{N}$:
\begin{align}\label{MKay}
  \ell^{m}_n  \widehat{M}^{r}\Big(\frac{ K  }{  \ell^{m}_n  }\Big)\,  > & \, K\bigg(1+\frac{b-1}{2} \frac{K}{\ell^{m}_n }   \bigg)^r  \,,\nonumber
\intertext{and as long as  $\frac{(b-1)K}{2\ell^{m}_n}$ is smaller than the solution $x>1$ to the equation    $1+x =\exp\{ x/2   \}  $, the above is larger than}
   >&\,   K\exp\left\{  \frac{r}{\ell^{m}_n }\frac{b-1}{4} K\right\} \,.
\end{align}

Notice that the sequence $a_{r}:=  \ell^{m}_n  \widehat{M}^{r}\big(\frac{ K  }{  \ell^{m}_n }\big)$ will be $\geq 2K$ before $r\in \mathbb{N}$ reaches the value $\lceil \frac{4\ell^{m}_n}{(b-1)K}  \rceil $.  Define the sequence $t_{j}\in \mathbb{N}$ by
$$ t_{j}\, :=\, \sum_{i=1}^{j} \left\lceil \frac{4\ell^{m}_n}{2^{i-1}(b-1)L}  \right\rceil  \, .  $$
By applying~(\ref{MKay}) inductively, we have that $  \ell^{m}_n  \widehat{M}_{n}^{t_{j}}\big(\frac{ L  }{  \ell^{m}_n  }\big)\,>\, L 2^{j}$, and thus
\begin{align}\label{Gobble}
\widehat{M}^{\alpha_{n,\epsilon}^{(m)}}\Big(\frac{ L  }{  \ell^{m}_n  }\Big)\, > \, 2^{N_{L,\epsilon}  }\frac{ L }{ \ell^{m}_n }\,,
\end{align}
where $N_{L,\epsilon}$ is the number of $ t_{j}$'s less than $\alpha_{n,\epsilon}^{(m)}$.   However,  $N_{L,\epsilon}$ grows   in rough proportion to $\ell^{m}_n $:
$$  \frac{N_{L,\epsilon}}{ \ell^{m}_n } \, \approx \, \frac{\alpha_{n,\epsilon}^{(m)}-  \frac{8}{(b-1)L}\ell^{m}_n   }{\ell^{m}_n } \,>\,\frac{1}{2\kappa_{b}^2}\left(1-\frac{\eta_{b}}{\epsilon}\right)\,,$$
where the second inequality uses  that $L\geq \frac{16\epsilon }{(\epsilon -\eta_{b})(b-1)}$. Hence, $2^{N_{L,\epsilon}  }\frac{ L}{ \ell^{m}_n }$ grows without bound as $n\rightarrow \infty$.   Combined with~(\ref{Dimp}) and~(\ref{Gobble}), this implies that
the variance of $ W_{n}\big( \beta_{n, \epsilon}^{(m)}  \big)$ goes to infinity.

\end{proof}

\section{The limit theorem \label{SecLimitThm}}

In this section we prove the central limit theorem in part (ii) of Theorem~\ref{ThmMain}. This requires extra notation for the diamond graph and  recursive formulae related to the normalized partition function \vspace{.3cm}

  For easy reference, we make the following notation list in which $k < n$:
\begin{align*}
&D_{n}   &   &  \text{$n^{th}$ diamond graph }            \\
&E_{n}   &   & \text{Set of edges on $D_{n}$}         \\
&E_{k}   &   & \text{Abusing notation, $E_{k} $ is identified with the set of copies of $D_{n-k}$ on $D_{n}$   }          \\
&a\Try g    &     & \text{$a\in E_{n}$ lies ``on" $g\in E_{k}$ }\\
&g\TS (i,j) &    &   \text{Refers to an element $E_{k+1}$ given $g\in E_{k}$ and $1\leq i,j\leq b$.  }
\end{align*}

The edge set contains $ |E_{n}|=b^{2n}  $ elements.  The inductive nature of the construction of the diamond graphs implies that there is a canonical one-to-one correspondence between the edge set $E_{n-k}$ and the set of copies of $D_{k}$ embedded in $D_{n}$.   Elements of $E_{1}$ can be labeled by $(i,j) \in \{1,\cdots,\bbf   \}\times \{1,\cdots, \bbf \}$, where $(i,j)$ refers to the $j^{th}$ segment on the $i^{th}$ branch.  Moreover,  we can label elements in $E_{k+1}$ using elements of $E_{k}$ via the correspondence
\begin{align*}
&E_{k+1} \, \equiv \, E_{k}\times \underbrace{(\{1,\cdots,\bbf   \}\times \{1,\cdots, \bbf \})}\,.\\
& \hspace{.7cm} \big(\text{Labels the edge set for a local copy of $D_{1}$}\big)
\end{align*}
Using this correspondence inductively, $E_{k}$ defines a partition of $E_{n}$ when $k<n$, and we will write $a\Try g$ for $a\in E_{n}$ and $g\in E_{k}$ when $a$ has the form $g\TS (i_{1},j_{1})\TS\cdots  (i_{n-k},j_{n-k})$ for some $1\leq i_{m},j_{m}\leq b$. 

\begin{remark}
For the remainder this article, $n\in \mathbb{N}$ will always refer to the size of the system and $k\in \mathbb{N}$ satisfies $k\leq n$.
\end{remark}

\begin{definition} Let $1\leq k \leq n$, $g\in E_{k}$, and $\beta >0$.

\begin{itemize}
\item  $W_{n}(\beta; g)$ is defined in analogy to  $W_{n}(\beta)$ except restricted to the diamond subgraph associated with $g\in E_{k}$.

\item  $ R_{n}(\beta; g)\, : =\,  W_{n}\big(\beta; g\big)-1  $

\item  $\displaystyle R_{k,n}(\beta)\, :=\,\frac{1}{b^{n-k}}\sum_{ g\in E_{n-k} }R_{n}(\beta; g)$
\end{itemize}

\end{definition}

Recall that $\mathbf{E}(\beta; g)$ is defined as in~(\ref{ExpWII}) for $g\in E_{n}$. Note that since $|E_{n-k}|=b^{2(n-k)}$ and the $R_{n}(\beta; g)$ are i.i.d.\  (as indexed by $g$) we have \[\textup{Var}\big( R_{k,n}(\beta)\big) \, = \, \textup{Var}\big(  R_{n}(\beta; g)\big)\,=:\,\varrho_{k}(\beta).\]

\begin{lemma} \label{LemSimple} Let $g\in E_{k}$ for $1\leq k\leq n$.
\begin{enumerate}[(i)]

\item  The family of random variables $R_{n}(\beta; g)$ satisfies the recursive relation
\begin{align*}
R_{n}(\beta; g)\, =\,  \frac{1}{b}\sum_{i=1}^{b}\Bigg[\prod_{1\leq  j\leq b} \Big(1+ R_{n}\big(\beta; g\TS (i,j)\big)\Big)\,-\, 1\Bigg]
\end{align*}
with initial condition $R_{n}(\beta; g)= \mathbf{E}\big( \beta; g\big)-1$ for $g\in E_{n}$.

\item For  $0\leq j  < k\leq n$, the random variables $R_{j,n}( \beta )$ and  $R_{k,n}( \beta)- R_{j,n}( \beta)$ are uncorrelated.

\end{enumerate}

\end{lemma}

\begin{remark}\label{TrivyDivy}
In particular this recursion implies that $R_{n,n}(\beta; g) = W_n(\beta)$.
\end{remark}

\subsection{Proof of part (ii) of Theorem~\ref{ThmMain}}

We will need some control of the fourth moment of $ R_{n}( \beta; g)$ in terms of its second moment, in order to apply a Lindeberg condition in the proof of Theorem~\ref{ThmMain}.

\begin{lemma}\text{ }\label{LemFourth}For $g\in E_{n-k}$ define  $\varrho_{k,v}( \beta) \,:= \, \mathbb{E}\big[ ( R_{n}( \beta; g))^{v}\big]$.    For any $m,r\in \mathbb{N}$ and  $\epsilon <\eta_{b}$, there is a $C>0$ such that for all $n>0$ and $0<k\leq n$
$$\varrho_ {k,2r}\big(\beta_{n, \epsilon}^{(m)}\big)\, \leq  \, C \Big(\varrho_ {k,2}\big(\beta_{n, \epsilon}^{(m)}\big)\Big)^r \,.    $$

\end{lemma}
\begin{proof}      Suppose for the purpose of a strong induction argument in $t=2,3,\cdots$  that there is a $c>0$ such that for all $k$, $n$, $v$ with $k\leq n$ and $v< t$
\begin{align}\label{Duzzle}
\big|\varrho_ {k,v}\big(\beta_{n, \epsilon}^{(m)}\big)\big|\, \leq  \, c \Big(\varrho_ {k,2}\big(\beta_{n, \epsilon}^{(m)}\big)\Big)^{\frac{v}{2}} \,.
\end{align}
Note that the base case $t=2$ of the above statement holds with $c=1$ by Jensen's inequality.  The recursive relation in part (i) of Lemma~\ref{LemSimple} implies that $k\mapsto \varrho_ {k, t}\big(\beta_{n, \epsilon}^{(m)}\big) $ obeys a recursive inequality of the form
\begin{align}
 \big|\varrho_ {k+1, t}\big(\beta_{n, \epsilon}^{(m)}\big) \big|\,\leq \,&\frac{1}{b^{t-2}}\big|\varrho_ {k, t}\big(\beta_{n, \epsilon}^{(m)}\big)\big|\,+\,
P\Big(  \big|\varrho_ {k, j}\big(\beta_{n,\epsilon}^{(m)}\big)\big| ;\, 2\leq j\leq t  \Big)  \big|\varrho_ {k, t}\big(\beta_{n, \epsilon}^{(m)}\big) \big| \nonumber \\  &\,+\,Q\Big(  \big|\varrho_ {k, j}\big(\beta_{n,\epsilon}^{(m)}\big)\big| ;\, 2\leq j\leq t -1     \Big)  \,, \label{Hugh}
\end{align}
where $P$ is a polynomial with no constant term and $Q\big(  \big|\varrho_ {k, j}\big(\beta_{n,\epsilon}^{(m)}\big)\big| ;\, 2\leq j\leq t -1     \big)$ is a linear combination of products
 $$\big|\varrho_ {k, j_{1}}\big(\beta_{n,\epsilon}^{(m)}\big)\big|\, \big|\varrho_ {k, j_{2}}\big(\beta_{n,\epsilon}^{(m)}\big)\big|\,\cdots \, \big|\varrho_ {k, j_{\ell}}\big(\beta_{n,\epsilon}^{(m)}\big)\big|$$
with $j_{i}< t$  and  $j_{1}+\cdots+j_{\ell}\geq t $.  Since the variable $\omega$ has finite exponential moments,  for $k=0$ we have  $$\big|\varrho_ {0, t}\big(\beta_{n, \epsilon}^{(m)}\big)\big| \,=\,\Big|\mathbb{E}\Big[\big( \exp\big\{\omega \beta_{n, \epsilon}^{(m)} \big\}-1   \big)^{t}\Big]\Big| =\mathit{O}\Big(\frac{1}{n^{\frac{t}{2}}}   \Big)$$ is small with large $n$.

If $ \big|\varrho_ {k, t}\big(\beta_{n, \epsilon}^{(m)}\big) \big|   \leq \lambda $ for $\lambda <1$,  the factor $P\big(  \big|\varrho_ {k, j}\big(\beta_{n,\epsilon}^{(m)}\big)\big| ;\, 2\leq j\leq t  \big)$ in~(\ref{Hugh}) has a bound of the form
\begin{align*}
P\Big(  \big|\varrho_ {k, j}\big(\beta_{n,\epsilon}^{(m)}\big)\big| ;\, 2\leq j\leq t  \Big)\,\leq \, c\Big(\lambda\,+\, \frac{1}{\ell^{m}_n}\Big)\,
\end{align*}
for some $c>0$ and all $\lambda<1$ and $n$.  The above holds by the induction assumption~(\ref{Duzzle}) and because $\varrho_ {n, 2}\big(\beta_{n, \epsilon}^{(m)}\big)$ is uniformly bounded by a constant multiple of $1/\ell^{m}_n$ for all $n\geq 1$  as a consequence of Lemma~\ref{LemVarConv}.  Pick $\lambda\ll 1$ and let $\widehat{k}\in \mathbb{N}$ be the smallest value such that $\big| \varrho_ {\widehat{k}, t}\big(\beta_{n, \epsilon}^{(m)}\big)\big|> \lambda $.  The term $Q\big(  \big|\varrho_ {k, j}\big(\beta_{n,\epsilon}^{(m)}\big)\big| ;\, 2\leq j\leq t -1     \big) $  is bounded by a constant multiple of $\big[\varrho_{k,2}\big(\beta_{n, \epsilon}^{(m)}\big)\big]^{t/2}$ by~(\ref{Duzzle}).   By the above considerations, there is an $\delta\in (0,1)$ and a $C>0$ such that for all $k<\widehat{k}$ and $n\in \mathbb{N}$
\begin{align}\label{BowWow}
 \big| \varrho_ {k+1, t}\big(\beta_{n, \epsilon}^{(m)}\big)\big| \, \leq  \, \delta  \varrho_ {k, t}\big(\beta_{n, \epsilon}^{(m)}\big)\,+\,C\big[\varrho_{k,2}\big(\beta_{n, \epsilon}^{(m)}\big)\big]^{\frac{t}{2}}\,.
\end{align}
   Using~(\ref{BowWow}) recursively,  it follows that for $k<\widehat{k}$
\begin{align}
 \big|\varrho_ {k, t}\big(\beta_{n, \epsilon}^{(m)}\big)\big| \, \leq  & \, \delta^{k}\varrho_ {0, t}\big(\beta_{n, \epsilon}^{(m)}\big)\,+\,C \sum_{j=0}^{k-1}\delta^{k-1-j}\big[\varrho_{j,2}\big(\beta_{n, \epsilon}^{(m)}\big)\big]^{\frac{t}{2}}\nonumber \\  \,\leq & \, \mathit{O}\Big(\frac{1}{n^{\frac{t}{2}}}   \Big)  \,+\,\frac{C}{1-\delta}\big[\varrho_{k,2}\big(\beta_{n, \epsilon}^{(m)}\big)\big]^{\frac{t}{2}}\,\leq C'\big[\varrho_{k,2}\big(\beta_{n, \epsilon}^{(m)}\big)\big]^{\frac{t}{2}}\,.
\end{align}
Since $\sup_{1\leq k\leq n}\varrho_{k,2}\big(\beta_{n, \epsilon}^{(m)}\big)=\varrho_{n,2}\big(\beta_{n, \epsilon}^{(m)}\big)\propto \frac{1}{\ell^{m}_{n}} $ is small with large $n$, it follows that $\widehat{k}>n$ when $n$ is large enough. 
Therefore, $\displaystyle \varrho_ {k, t}\big(\beta_{n, \epsilon}^{(m)}\big)$ is bounded by a multiple of $\big|\varrho_{k,2}\big(\beta_{n, \epsilon}^{(m)}\big)\big|^{\frac{t}{2}}$ for all $k$, $n$ with $k\leq n$, and the induction  step is complete.


\end{proof}

\vspace{.3cm}

\begin{proof}[Proof of part (ii) of Theorem~\ref{ThmMain}]  Let $(u_{n})_{n\geq 0}$ be a non-decreasing sequence of integers with $1\ll  u_{n} \ll  \ell^{m}_n   $.   Note that $  R_{ n}\big(\beta_{n, \epsilon}^{(m)}\big) $ and $R_{n ,n}\big(\beta_{n, \epsilon}^{(m)} \big)   $ are equal.   It suffices to prove:
\begin{enumerate}[I)]

\item $ \big( \ell^{m}_n \big)^{\frac{1}{2}} \Big[ R_{n ,n}\big(\beta_{n, \epsilon}^{(m)} \big) \, -\, R_{n-u_{n} ,n}\big(\beta_{n, \epsilon}^{(m)} \big)\Big] \hspace{.2cm} \stackrel{\mathcal{P}}{\Longrightarrow} \hspace{.2cm} 0\, $

\item  $\ds\big( \ell^{m}_n\big)^{\frac{1}{2}} R_{n-u_{n} ,n}\big(\beta_{n, \epsilon}^{(m)} \big) \hspace{.2cm} \stackrel{\mathcal{L}}{\Longrightarrow}  \hspace{.2cm}\mathcal{N}\big(0,\upsilon_{b}(\epsilon)\big)$\,

\end{enumerate}
The purpose of introducing $R_{n-u_{n} ,n}\big(\beta_{n, \epsilon}^{(m)} \big) $ as an approximation to $R_{n ,n}\big(\beta_{n, \epsilon}^{(m)} \big) $ is that $R_{n-u_{n} ,n}\big(\beta_{n, \epsilon}^{(m)} \big) $ is a sum of i.i.d.\ random variables--see~(\ref{Gabel}) below--and can be understood through the Lindeberg-Feller central limit theorem.\vspace{.5cm}

For (I)  note that by part (ii) of Lemma~\ref{LemSimple}, $R_{n-u_{n} ,n}\big(\beta_{n, \epsilon}^{(m)} \big)$ and   $R_{n ,n}\big(\beta_{n, \epsilon}^{(m)} \big) \, -\, R_{n-u_{n} ,n}\big(\beta_{n, \epsilon}^{(m)} \big)$ are uncorrelated, and thus
\begin{align}\label{Hornf}
  \mathbb{E}\bigg[\Big|   R_{ n ,n}\big( \beta_{n, \epsilon}^{(m)} \big)  \,  -\, R_{ n-u_{n} ,n}\big(\beta_{n, \epsilon}^{(m)}  \big)     \Big|^{2}\bigg]\,  =& \,\varrho_{n}\big(\beta_{n, \epsilon}^{(m)} \big)\,  - \, \varrho_{n-u_{n}}\big(\beta_{n, \epsilon}^{(m)} \big)\nonumber  \\
=&\, M^{u_{n}}\Big(\varrho_{n-u_{n}}\big(\beta_{n, \epsilon}^{(m)} \big)\Big)\,-\,\varrho_{n-u_{n}}\big(\beta_{n, \epsilon}^{(m)} \big)\nonumber
\\   \,  =& \,  \mathit{O}\bigg(\frac{u_{n} }{|\ell^{m}_n|^2 }\bigg)\, .
\end{align}
The order equality in~(\ref{Hornf}) holds for $n\gg 1$ since
\begin{align}
&\varrho_{n-u_{n}}\big(\beta_{n, \epsilon}^{(m)} \big)\,\leq \,\underbrace{\varrho_{n}\big(\beta_{n, \epsilon}^{(m)} \big)\,\sim\, \frac{\upsilon_{b}(\epsilon)}{ \ell^{m}_n }  }\,,\nonumber \\
&\hspace{2.9cm}\text{(by Lemma~\ref{LemVarConv})}   \nonumber
\end{align}
and because the  map $M$ has a bound of the form  $M(x)\,\leq  \,  x\big(1+  C/\ell^{m}_n     \big)$ for some $C>0$  independent of $n\in \mathbb{N}$ and $ x$  in the shrinking intervals $ [0, 2\upsilon_{b}(\epsilon)/ \ell^{m}_n   ]$.  Thus, for $x=\varrho_{n-u_{n}}\big(\beta_{n, \epsilon}^{(m)} \big)$

$$ M^{u_{n}}(x)\,-\,x\, \leq \, x\bigg(1+ \frac{C}{\ell^{m}_n}     \bigg)^{u_{n}}\,-\,x  \, \leq \, x\Big(\exp\Big\{C\frac{ u_{n} }{ \ell^{m}_n }   \Big\}-1   \Big) \,\leq \, \mathit{O}\Big(x \frac{ u_{n} }{ \ell^{m}_n }   \Big)\,.  $$
Therefore (I) holds.

  \vspace{.4cm}

For (II) recall that by definition of the random variable $R_{n-u_{n},n}\big(\beta_{n, \epsilon}^{(m)} \big) $,
\begin{align}\label{Gabel}
 \,\big(\ell^{m}_n\big)^{\frac{1}{2}} R_{n-u_{n},n}\big(\beta_{n, \epsilon}^{(m)} \big)  \, = \, \frac{ 1}{b^{u_{n}}} \sum_{ g\in E_{u_{n}}   }\big(\ell^{m}_n\big)^{\frac{1}{2}}R_{n}\big(\beta_{n, \epsilon}^{(m)}  ; g\big)  \,.
\end{align}
 Note that by~(\ref{Hornf}), we have the second equality below:
\begin{align*}
\textup{Var}\Big(\big(\ell^{m}_n\big)^{\frac{1}{2}}  R_{n}\big(\beta_{n, \epsilon}^{(m)}  ; g\big) \Big)\,=\,\ell^{m}_n\varrho_{n-u_{n}}\big(\beta_{n, \epsilon}^{(m)}\big)\,&=\,\underbrace{ \ell^{m}_n\varrho_{n}\big(\beta_{n, \epsilon}^{(m)}\big)}\,+\, \mathit{O}\Big( \frac{ u_{n}  }{ \ell^{m}_n  }  \Big) \,,\\
& \hspace{.45cm}\,\,\,\stackrel{n\rightarrow \infty}{\longrightarrow} \upsilon_{b}(\epsilon)
\end{align*}
where the convergence of $ \ell^{m}_n\varrho_{n}\big(\beta_{n, \epsilon}^{(m)}\big)$ to $\upsilon_{b}(\epsilon) $ as $n \to \infty$ follows from Lemma~\ref{LemVarConv}, as mentioned above.   Thus,  the right side of~(\ref{Gabel}) is a sum  of $|E_{u_{n}} |=b^{2u_{n}}$
i.i.d.\ random variables with mean zero and variance $\approx \upsilon_{b}(\epsilon) $. Therefore (II) holds by the Lindeberg-Feller central limit theorem, so long as for any fixed $\epsilon>0$ it can be shown that
\begin{align}\label{Pimpernel}
\frac{ 1}{b^{2u_{n}}} \ell^{m}_n \sum_{ g\in E_{u_{n}}   } \mathbb{E}\Big[  \big|  R_{n}\big(\beta_{n, \epsilon}^{(m)} ; g\big)\big|^2 \chi\Big( \big(\ell^{m}_n \big)^{\frac{1}{2}}\big|R_{n}\big(\beta_{n, \epsilon}^{(m)}; g\big)\big| > \epsilon b^{u_{n}} \Big)   \Big] \xrightarrow{n \to \infty} 0\, .
\end{align}
However, by Chebyshev, the expression above is smaller than
\begin{align*}
\frac{ 1}{\epsilon^2 b^{4u_{n} }}\big( \ell^{m}_n \big)^2 \sum_{ g\in E_{u_{n} }   } \mathbb{E}\Big[  \big|  R_{n}\big(\beta_{n, \epsilon}^{(m)} ; g\big)\big|^4  \Big] \,  = & \,   \frac{ 1}{\epsilon^2 b^{2u_{n} }}\big( \ell^{m}_n \big)^2\varrho_ {n-u_{n}, 4}\big(\beta_{n, \epsilon}^{(m)}\big)\,.\nonumber
\intertext{Since $\varrho_ {n-u_{n}, 4}\big(\beta_{n, \epsilon}^{(m)}\big)<\varrho_ {n, 4}\big(\beta_{n, \epsilon}^{(m)}\big)   $, Lemma~\ref{LemFourth} implies that for some $C>0$ and all $n\in \mathbb{N}$ }
 \leq &\,\frac{ C}{\epsilon^2 b^{2u_{n} }}\big( \ell^{m}_n\big)^2\Big(\varrho_ {n}\big(\beta_{n, \epsilon}^{(m)}\big)\Big)^2 \, .
\intertext{By Lemma~\ref{LemVarConv}, $\varrho_ {n}\big(\beta_{n, \epsilon}^{(m)}\big)=\mathit{O}\big( 1/ \ell^{m}_n \big)$ and thus for some  $C'>0$}
 \leq &\,\frac{ C'}{\epsilon^2 b^{2u_{n} }}\,\,\stackrel{n\rightarrow \infty}{\longrightarrow} \,\, 0 \,.
 \end{align*}
Hence (II) holds and the proof is complete.
\end{proof}

\section{Quenched free energy}\label{SecQuenched}

Given $\beta>0$ the \textit{quenched free energy}  is defined as the limit
$$p(\beta)\,:=\, \lim_{n\rightarrow \infty}\frac{1}{s^n}\mathbb{E}\big[\log \big(Z_{n}(\beta)   \big) \big] \,. $$
The existence of the limit is guaranteed by the same  argument as used by Lacoin and Moreno for the site disorder model on the diamond lattice in the beginning of~\cite[ Section 3]{lacoin}.  Also in the context of the site disorder model, \cite[Proposition 4.3]{lacoin} shows that   bounds for the variance of the normalized partition function, $W_{n}(\beta)$, can be used to obtain lower bounds for the quenched free energy.  The authors  mention in~\cite[Section 8]{lacoin} that their reasoning  can be applied to the bond disorder model in the $b=s$ case to show that  for some $c>0$ and all $\beta<1$
\begin{align}\label{Prev}
\lambda(\beta)\,-\,p(\beta)\,\leq \,\exp\Big\{ - \frac{c}{\beta^2} \Big\}\,,
\end{align}
where $\lambda(\beta):= \log\big(\mathbb{E}\big[ e^{\beta\omega } \big]   \big) $.  The variance convergence  in part (i) of  Theorem~\ref{ThmMain}  for $\epsilon<\eta_{b}$  combined with their argument yields the following refinement of~(\ref{Prev}):

\begin{theorem}\label{ThmQuenched}  Let $b=s$ and fix $N\in \mathbb{N}$.   The following inequality holds for some $C>0$ and small  $\beta >0$:
\begin{align*}
  \lambda(\beta)-p(\beta)\,\leq \, C \exp\Bigg\{   -\frac{\kappa_{b}^2\log b}{\beta^2}\,+\,\frac{\tau\kappa_{b}^2\log b}{\beta}\,-\,\eta_{b}\log b \Big( \log\Big(\frac{1}{\beta}\Big)-\log^{N+1}\Big(\frac{1}{\beta}\Big)        \Big) \Bigg\} \,,
\end{align*}
 where $\log^n$ is the $n$-fold composition of $\log$.
\end{theorem}

It is not fully clear that the above  is a meaningful ``refinement" without having corresponding lower bounds for $\lambda(\beta)-p(\beta)$, however, it is reasonable to expect that it is optimal based on the explosion of the variance beyond the critical point $\eta_{b}$ in part (i) of  Theorem~\ref{ThmMain}.  Finding  lower bounds  is an interesting problem that will require new techniques.

We will  make a few comments on  the proof of Theorem~\ref{ThmQuenched}, which follows from the reasoning in~\cite[Section 4.3]{lacoin} with~\cite[Lemma~4.2]{lacoin} replaced by the following corollary to part (i) of Theorem~\ref{ThmMain}:

\begin{corollary}\label{Cor}
Define $\gamma_{N,\beta}   : =  \Big\lceil \frac{  \kappa_{b}^2  }{ \beta^2 }-\frac{\tau\kappa_{b}^2}{\beta} + \eta_{b}  \sum_{j=1}^N \ell^{j}_{  1 /\beta^2  }  \Big\rceil  $.   For fixed $N\in \mathbb{N}$,  as $\beta\searrow 0$
$$   \textup{Var}\Big(W_{ \gamma_{N,\beta} } (  \beta  ) \Big) \quad  \longrightarrow \quad  0       \,.   $$
\end{corollary}

\begin{proof}Apply part (i) of Theorem~\ref{ThmQuenched} with $n=\gamma_{N,\beta} $ since $\beta< \beta_{n, \epsilon}^{(m+1)}  $ for any fixed $\epsilon\in (0,\eta_{b})   $ and small enough $\beta>0$.
\end{proof}

Define $p_{c}\in (0,1)$ as the solution to $p_{c}=1-\big(1-p_{c}^s\big)^{b}   $, i.e., the critical percolation probability for the diamond lattice. A small modification of  Lacoin and Moreno's argument implies the following:
\begin{proposition*}[Analog to \text{\cite[Proposition~4.3]{lacoin}}]   Let $n$ be an integer such that $v_{n}:=\textup{Var}\big(W_{n}(\beta)\big)<\frac{1-p_{c}}{4}$.  There is a $C>0$ such that
$$ \lambda(\beta)\,-\,p(\beta)\,\leq \, Cs^{-n}\,.$$

\end{proposition*}

\begin{proof}[Proof of Theorem~\ref{ThmQuenched}]
By Corollary~\ref{Cor} it is possible to pick $\beta$ small enough so that
\begin{align}\label{Duff}
  \textup{Var}\Big(W_{ \gamma_{N,\beta} } (  \beta  )\Big)  \,< \, \frac{1-p_{c}}{4}   \,.
\end{align}
  By~\cite[Proposition~4.3]{lacoin} with  $b=s$,
\begin{align*}
  \lambda(\beta)\,-\,p(\beta)\,\leq  \, Cb^{-\gamma_{N,\beta}}  \, = & \, C \exp\Bigg\{  -\log b \Bigg\lceil \frac{  \kappa_{b}^2  }{ \beta^2 }-\frac{\tau\kappa_{b}^2}{\beta}  + \eta_{b} \sum_{j=1}^N \ell^{j}_{  1/\beta^2  }  \Bigg\rceil   \Bigg\}   \, ,         
\end{align*}
which implies the result since
\begin{align*} 
 \Bigg\lceil \frac{  \kappa_{b}^2  }{ \beta^2 }-\frac{\tau\kappa_{b}^2}{\beta}  + \eta_{b} \sum_{j=1}^N \ell^{j}_{  1 /\beta^2  }  \Bigg\rceil     \,=\,   \frac{\kappa_{b}^2}{\beta^2}\,-\,\frac{\tau\kappa_{b}^2}{\beta}\,+\,\eta_{b} \Big( \log\Big(\frac{1}{\beta}\Big)-\log^{N+1}\Big(\frac{1}{\beta}\Big)        \Big)+\mathit{O}(1) \,.  
  \end{align*}

\end{proof}


\begin{thebibliography}{99}


\bibitem{US}  T.\ Alberts, J.\ Clark, S.\ Koci\'c: \emph{The intermediate disorder regime for a directed polymer model on a hierarchical lattice}, Stoch. Proc. Appl. \textbf{127}, No.\ 10, 3291-3330 (2017).

\bibitem{alberts} T. Alberts, K. Khanin, J. Quastel: \emph{The intermediate disorder regime for directed polymers in dimension $1+1$}, Ann. Probab. \textbf{42}, No. 3, 1212-1256 (2014).


\bibitem{Alberts_Ortgiese} T. Alberts, M. Ortgiese: \emph{The near-critical scaling window for directed polymers on disordered trees}, Electron. J. Probab. \textbf{18}, No. 19, 1-24 (2013).

\bibitem{Bolth} E. Bolthausen: \textit{A note on the diffusion of directed polymers in a random environment}, Comm. Math. Phys. \textbf{123} 529-534 (1989).


\bibitem{Carmona} P. Carmona, Y. Hu: \textit{Strong disorder implies strong localization for directed polymers in a random environment}, ALEA \textbf{2}, 217-229.


\bibitem{Carav2} F. Caravenna, R. Sun, and N. Zygouras: \textit{The continuum disordered pinning model}, to appear in Prob. Theor. Rel. Fields.


\bibitem{Comets} F. Comets, T. Shiga, and N. Yoshida: \textit{Probabilistic analysis of directed polymers in a random environment: a review}, Adv. Stud. Pure Math. \textbf{39}, 115-142 (2004).

\bibitem{Cook} J. Cook, B. Derrida: \emph{Polymers on disordered hierarchical lattices:  a nonlinear combination of random variables}, J. Stat. Phys. \textbf{57} 89-139 (1989).


\bibitem{Gardner} B. Derrida, E. Gardner: \emph{Renormalisation group study of a disordered model}, J. Phys. A: Math. Gen. \textbf{17}, 3223-3236 (1984).

\bibitem{Giacomin} B. Derrida, G. Giacomin, H. Lacoin, F.L. Toninelli: \emph{Fractional moment bounds and disorder relevance for pinning models}, Commun. Math. Phys. \textbf{287}, 867-887 (2009).

\bibitem{Derrida} B. Derrida, R.B. Griffith: \emph{Directed polymers on disordered hierarchical lattices}, Europhys. Lett. \textbf{8}, No. 2, 111-116 (1989).

\bibitem{Hakim} B. Derrida, V. Hakim, J. Vannimenius: \emph{Effect of disorder on two-dimensional wetting}, J. Stat. Phys. \textbf{66} 1189-1213 (1992).

\bibitem{DeyZyg} P.S. Dey, N. Zygouras: \textit{High temperature limits for $(1+1)$-dimensional directed polymer with heavy-tailed disorder}, Ann. Probab. \textbf{44}, No. 6, 4006-4048 (2016).

\bibitem{Flores} G.R.M. Flores, T. Sepp\"al\"ainen, B. Valk\'o: \emph{Fluctuation exponents for directed polymers in the intermediate disorder regime},  Electron. J. Probab. \textbf{19}, no. 89., 1-28 (2014).

\bibitem{Garel} T. Garel, C. Monthus: \emph{Critical points of quadratic renormalizations of random variables and phase transitions of disordered polymer models on diamond lattices}, Phys. Rev. E \textbf{77}, 021132 (2008).

\bibitem{GLT}  G. Giacomin, H. Lacoin, F.L. Toninelli: \emph{Hierarchical pinning models, quadratic maps, and quenched disorder}, Probab. Theor. Rel. Fields \textbf{145}, (2009).

\bibitem{Griffiths} R.B. Griffith, M. Kaufman: \emph{Spin systems on hierarchical lattices.  Introduction and thermodynamical limit}, Phys. Rev. B, \textbf{3} 26, no. 9, 5022-5032 (1982).

\bibitem{Hambly} B.M. Hambly, J.H. Jordan: \emph{A random hierarchical lattice: the series-parallel graph and its properties}, Adv. Appl. Prob., \textbf{36}, 824-838 (2004).

\bibitem{HamblyII}  B.M. Hambly, T. Kumagai: \emph{Diffusion on the scaling limit of the critical percolation cluster in the diamond hierarchical lattice}, Adv. Appl. Prob., \textbf{36}, 824-838 (2004).


\bibitem{Imbr} J.Z. Imbrie, T. Spencer:  \emph{Diffusion of directed polymers in a random environment}, J. Stat. Phys. \textbf{52}, 609-622 (2001).

\bibitem{lacoin} H. Lacoin, G. Moreno:  \emph{Directed Polymers on hierarchical lattices with site disorder}, Stoch. Proc. Appl. \textbf{120}, No. 4, 467-493 (2010).

\bibitem{lacoin3} H. Lacoin: \emph{Hierarchical pinning model with site disorder: disorder is marginally relevant}, Probab. Theor. Rel. Fields \textbf{148}, No. 1-2, 159-175 (2010).

\bibitem{lacoin2} H. Lacoin:  \emph{New bounds for the free energy of directed polymers in dimension $1+1$ and $1+2$},  Commun. Math. Phys. \textbf{294}, No. 2, 471-503 (2010).



\bibitem{Sepp} T. Sepp\"al\"ainen: \emph{Scaling for a one-dimensional directed polymer with boundary conditions},  Ann. Probab. \textbf{40}, No. 1, 19-73 (2012).

\bibitem{Spohn} T. Schl\"osser, H. Spohn: \emph{Sample to sample fluctuations in the conductivity of a disordered medium}, J. Statist. Phys. \textbf{69}, 955-967 (1992).


\bibitem{Vargas} V. Vargas: \emph{Strong localization and macroscopic atoms for directed polymer}, Probab. Theor. Rel. Fields \textbf{134}, 391-410 (2008).


\bibitem{Wehr}  J. Wehr, J.M. Woo: \emph{Central limit theorems for nonlinear hierarchical sequences or random variables}, J. Statist. Phys. \textbf{104}, 777-797 (2001).

\end{thebibliography}
\end{document}